\documentclass[11pt, notitlepage]{article}
\usepackage{amssymb,amsmath,comment}
\usepackage{epsfig,enumerate,amsmath,amsfonts,amssymb}
\usepackage{indentfirst}
\usepackage{pstricks}
\usepackage{setspace}
\usepackage{latexsym}
\usepackage{amsthm}
\usepackage{amsmath}
\usepackage{parskip}
\usepackage{fixmath}
\catcode`\@=11 \@addtoreset{equation}{section}
\def\thesection{\arabic{section}}

\def\theequation{\thesection.\arabic{equation}}
\catcode`\@=12
\usepackage{colortbl}
\usepackage{a4wide}

\newcommand{\ds} {\displaystyle}
\newcommand{\e}{\varepsilon}
\newcommand{\pa} {\partial}
\newcommand{\al} {\alpha}

\newcommand{\de} {\delta}

\newcommand{\Om} {\Omega}
\newcommand{\ra} {\rightarrow}
\newcommand{\rp} {\rightharpoonup}
\newcommand{\De} {\Delta}
\newcommand{\la} {\lambda}
\newcommand{\La} {\Lambda}
\newcommand{\noi} {\noindent}
\newcommand{\na} {\nabla}
\newcommand{\mb} {\mathbb}

\newcommand{\lra} {\longrightarrow}
\newcommand{\ld} {\langle}
\newcommand{\rd} {\rangle}

\usepackage[all]{xy}
\catcode`\@=11
\def\theequation{\@arabic{\c@section}.\@arabic{\c@equation}}
\catcode`\@=12

\def\QED{\hfill {$\square$}\goodbreak \medskip}

\newtheorem{Theorem}{Theorem}[section]
\newtheorem{Lemma}[Theorem]{Lemma}
\newtheorem{Proposition}[Theorem]{Proposition}

\newtheorem{Remark}[Theorem]{Remark}

\begin{document}
\vspace{0.01in}

\title
{Kirchhoff equations with Hardy-Littlewood-Sobolev critical nonlinearity}

\author{ {\bf Divya Goel\footnote{e-mail: divyagoel2511@gmail.com} \; 
	and \;  K. Sreenadh\footnote{
		e-mail: sreenadh@maths.iitd.ac.in}} \\ Department of Mathematics,\\ Indian Institute of Technology Delhi,\\
	Hauz Khaz, New Delhi-110016, India. }

\date{}

\maketitle

\begin{abstract}
We consider the following Kirchhoff - Choquard equation
\[
-M(\|\na u\|_{L^2}^{2})\De u = \la f(x)|u|^{q-2}u+ \left(\int_{\Om}\frac{|u(y)|^{2^*_{\mu}}}{|x-y|^{\mu}}dy\right)|u|^{2^*_{\mu}-2}u  \; \text{in}\;
\Om,\quad 
 u = 0 \;    \text{ on } \pa \Om ,
\]
 where  $\Om$ is a  bounded domain in $\mathbb{R}^N( N\geq 3)$ with $C^2$ boundary,  $2^*_{\mu}=\frac{2N-\mu}{N-2}$, $1<q\leq 2$,  
and $f$ is  a continuous  real valued  sign changing function. When $1<q< 2$, using the method of Nehari manifold and Concentration-compactness Lemma,  we prove the  existence and multiplicity of positive solutions of the above problem. We also prove the existence of a positive solution when $q=2$ using the Mountain Pass Lemma.
\medskip
  
\noi \textbf{Key words:} Kirchhoff equation, Hardy-Littlewood-Sobolev critical exponent, Positive solution.

\medskip

\noi \textit{2010 Mathematics Subject Classification: 35A15, 35J60, 35J20.} 

\end{abstract}
\section{Introduction}
The purpose of this article is to investigate the existence and multiplicity of positive solutions of the  following critical growth Kirchhoff-Choquard equation 
\begin{align*}
(P_\la)
\left\{
-\left(  a+\e^p \left(\ds \int_{\Om} |\na u|^2~dx\right)^{\theta-1}\right)\De u\right.
&=\la f(x)|u|^{q-2}u+ \left(\ds \int_{\Om}\frac{|u(y)|^{2^*_{\mu}}}{|x-y|^{\mu}}dy\right)|u|^{2^*_{\mu}-2}u,  \, \text{in}\,
\Om,\\
u&=0 \, \text{ on } \pa \Om, 
\end{align*}
\noi where $\Om \subset\mathbb{R}^N (N\geq 3$) is a bounded domain with $C^2$ boundary,   $ \e>0$ is small enough, $0<\mu<N$, $1<q\leq 2$, $a, \la,p,\theta$ are positive real numbers such that $ p>N-2$ and $\theta \in [1,2^*_{\mu})$. Here   $2^*_{\mu}=\frac{2N-\mu}{N-2}$ is the critical exponent in the sense of Hardy-Littlewood-Sobolev inequality (see \eqref{co9}). The function $f(x)$ is  a  continuous  real valued  sign changing function such that  $f\in L^{r}
(\Om),$ where $r=\frac{2^*}{2^*-q}, 2^*=2_0^*,$ is the critical exponent of the Sobolev embedding  $H^1_0(\Om)$ into $L^{2^*}(\Om)$. 

Recently,  the study of existence and uniqueness of positive solutions for Choquard type equations attracted a lot of attention of researchers due to its vast applications in physical models.  In 1954, Pekar\cite{pekar} studied the following  equation that arises in quantum theory of poloron: 
\begin{equation}\label{kh20}
-\De u +u = (|x|^{-1}* |u|^2) u\; \text{in}\; \mathbb{R}^3. 
\end{equation}
Later \eqref{kh20} was used  as an approximation of the equation that arises in Hartree-Fock theory\cite{ehleib}. Recently, Moroz and Schaftingen \cite{moroz4} studied the Choquard equations and proved the existence, asymptotic behavior  and symmetry of solutions. We cite  \cite{moroz2, moroz3} for the work of Choquard type equations over the whole space $\mathbb{R}^N$. In  \cite{yang},   Gao and  Yang  studied the  Brezis-Nirenberg type existence results for the following  critical Choquard problem in bounded domains $\Om\subset \mb R^N, N\ge 3$ having smooth boundary $\pa \Om$:
\begin{equation*}
-\De u = \la h(u)+ \left(\int_{\Om}\frac{|u(y)|^{2^*_{\mu}}}{|x-y|^{\mu}}dy\right)|u|^{2^*_{\mu}-2}u \text{ in } \Om, \quad
u=0 \text{ on } \pa \Om, 
\end{equation*}
\noi where $\la>0$, $0<\mu<N$ and $h(u)=u$.  Later  in \cite{yangjmaa} author used variational methods to prove the existence and multiplicity  of positive solutions for equations involving convex and convex-concave type nonlinearities ($h(u)=u^q, 0<q<1$). For more work on Choquard equations, Interested readers are   referred to \cite{guide,ts}  and references therein.\\
 On a similar note, the study of Kirchhoff-type equations received much attention due to its widespread application in  various models of physical  and biological systems. Indeed, Kirchhoff in \cite{kirchhoff}
	studied the following equation 
	\begin{align*}
	\rho \frac{\pa^2 u}{\pa t^2} -\left(\frac{P_0}{h}+\frac{E}{2L} \int_0^L \bigg| \frac{\pa u}{\pa x}  \bigg|^2~dx \right)\frac{\pa^2 u}{\pa x^2}=0,
	\end{align*}
	where $\rho, P_0,h,E,L$ represents physical quantities.  This model extends the classical D'Alembert wave equation by considering the effects of the changes in the length of the strings during the 	vibrations. Existence of solutions for Kirchhoff equations involving the critical Sobolev exponent have been studied by many authors. Chen, Kuo and Wu \cite{wu1}  studied the following Kirchhoff problem
	\begin{align*}
	-M(\|\na u\|_{L^2}^{2})\De u = \la f(x)|u|^{q-2}u+ g(x)|u|^{p-2}u  \; \text{in}\;
	\Om,\; 
	u = 0 \;    \text{ on } \pa \Om,
	\end{align*}
	where  $M(t)=a+b\,t$, $a,b>0,\; 1<q<2<p<2^*$ and   $f$ and $g $ are continuous  real valued  sign changing functions. Here authors proved the existence and multiplicity of solutions using the classical Nehari manifold methods. Recently, Lei, Liu and  Guo \cite{chunlei}, studied the following critical exponent problem 
	\begin{equation}\label{kh19}
	-\left(  a+\e \|\na u\|_{L^2}^{2} \right)\De u
	=\la |u|^{q-2}u+|u|^{4} u\; \text{in}\;
	\Om,
	u=0 \; \text{ on } \pa \Om,
	\end{equation}
	where $\Om$ is a bounded domain in $\mathbb{R}^3,\; a>0,\; 1<q<2,\; \e>0$ is small enough and $\la>0$ is positive real number. Here they proved that if $\e>0$ is sufficiently small then there exists a $\la_*>0$ such that for any $\la\in (0,\la_*),$ problem \eqref{kh19} has at least two positive solutions, and one of the solution is a  ground state solution.	 We refer to \cite{alves1, pawan1,valdi1,  Figueiredo1, Figueiredo2, pawan3} for Kirchhoff problems involving the classical Laplace operator and  $p-$fractional Laplace operators. 
 
In \cite{lu}, L\"u studied the following Kirchhoff equation with Hartee-type nonlinearity 
\begin{align}\label{kh21}
-\left(  a+b \|\na u\|_{L^2}^{2} \right)\De u + (1+\la g(x))u 
=\left(\ds |x|^{-\mu}* |u|^{p} \right)|u|^{p-2}u, \; \text{in}\; \mathbb{R}^3
\end{align}
where $a>0$, $ b\geq 0$ are constants, $\mu \in  (0, 3)$, $p \in  (2, 6 -\mu )$, $\la >0$ is a parameter and $g(x)$ is a nonnegative continuous potential satisfying some conditions. By using the technique of  Nehari manifold and the concentration compactness principle, authors proved the  existence of ground state solutions of \eqref{kh21}, if the parameter $\la$ is large enough. Later  Li, Gao and Zhu \cite{li}, studied the  existence of  
sign-changing solutions to a class of Kirchhoff-type systems with Hartree type nonlinearity
in $\mathbb{R}^3$ on the sign-changing Nehari manifold and a quantitative deformation lemma.
 All the above mentioned articles on  Choquard-k
 Kirchhoff problems are on $\mathbb{R}^3$. To the  best of our knowledge, there is no result available in the current literature on   Kirchhoff equations
 with Choquard nonlinearity in higher dimension.\\
 
 In  this article we consider the Choquard-Kirchhoff problems  with critical growth nonlinearity in higher dimensions. We study the existence and multiplicity of positive solutions of the problem $(P_\la)$. Using the variational methods on the Nehari manifold we prove the existence of two positive solutions. For the existence of first solution we use the minimization argument over the Nehari manifold associated with problem $(P_\la)$.  In order to prove the existence of second solution we divide the proof into two cases: $\mu<\min \{4,N\}$ and $\mu\geq \min \{4,N\}$.  The salient feature of this article is the  novel asymptotic analysis (See Lemma \ref{khlem9} and Lemma \ref{khlem10})  to study the critical level below which Palias-Smale sequences are compact.  The asymptotic estimates on the critical term are delicate and we use various inequalities especially when $2_{\mu}^{*} \in (2, 3)$. Finally, by finding a relation between $\la$ and $\e$ we obtain the required sequence below the critical level. We also proved the existence of a positive solution of $(P_\la)$ in case of $q=2$ using the Mountain Pass Lemma.  Overall, this work adds to the body of knowledge and is a new contribution to the literature of Choquard-Kirchhoff equations.  With this introduction we will state our main results:

%
\begin{Theorem}\label{khthm1}
Let $1<q< 2$ then there exists  $\La^*>0 $ such that for all $\la \in (0, \La^*)$,  $(P_\la)$ admits a  positive solution  for all $\e>0$.
\end{Theorem}

\begin{Theorem}\label{khthm2}
	  Let $1<q< 2$ then there exist $\Upsilon^* ,\; \Upsilon^{**} >0 $ and $\e^*,\; \e^{**}>0$ such that 
	\begin{enumerate}
		\item [(i)] if $\mu<\min \{4,N\},\; \la \in (0, \Upsilon^*)$ and $\e \in (0, \e^*)$, then  $(P_\la)$ admits at least two positive solutions. 
	\item [(ii)] if $\mu\geq \min \{4,N\}, \; \la \in (0, \Upsilon^{**})$, $\e \in (0, \e^{**})$ and $\frac{N}{N-2}\leq q<2$, then  $(P_\la)$ admits at least two positive solutions. 
	\end{enumerate}
\end{Theorem}
\begin{Theorem}\label{khthm3}
	Let $q=2$. Then there exists  $\tilde{\e}>0$ such that for any $\la \in(0,\;aS\|f\|_{L^r}^{-1}) $ and $\e \in (0,\tilde{\e})$, problem $(P_\la)$ has a positive solution.
\end{Theorem}
\begin{Remark}
	We remark that  the approach used in this paper can be applied for the following critical exponent problem 
	\[
	(Q_\la)\;
	-\left(  a+\e^p \|\na u\|_{L^2}^{2\theta-2} \right)\De u 
	=\la f(x)|u|^{q-2}u+ |u|^{2^*-2}u,  \;\; \text{in}\;
	\Om,\;
	u=0 \; \text{ on } \pa \Om,
	\]
	where,    $ \e>0$ is small enough, $1<q<2$, $a, \la,p,\theta$ are positive real numbers such that $ p>N-2$ and $\theta \in [1, 2^*/2)$ and $f$ is  a  continuous  real valued  sign changing function such that  $f\in L^{\frac{2^*}{2^*-q}}(\Om)$. Using the  methodology of \cite{wu2008} and  asymptotic analysis done in Lemma \ref{khlem9}, one can show the following result:
	\begin{Theorem}
		There exist $\Upsilon^* >0 $ and $\e^*>0$ such that the equation $(Q_\la)$ admits at least two positive solutions for all $\la\in(0,\Upsilon^*)$ and $\e\in(0,\e^*).$
	\end{Theorem}
	
\end{Remark}
Turing to layout of the article,  
 in section 2, we will give the variational framework, fibering map analysis and compactness of Palais-Smale sequences. In section 3, we have proved the existence of first  positive solution. In section 4, we have proved the existence of second positive solution.  In section 5, we prove the existence a positive solution when $q=2$. 

\section{Variational Framework and fibering map analysis}
Firstly we will give the variational framework of the problem $(P_\la)$. We start with the well - known  Hardy-Littlewood-Sobolev Inequality: 
\begin{Proposition}\cite{leib}
	Let $t,r>1$ and $0<\mu <N$ with $1/t+\mu/N+1/r=2$, $f\in L^t(\mathbb{R}^N)$ and $h\in L^r(\mathbb{R}^N)$. There exists a sharp constant $C(t,r,\mu,N)$ independent of $f,h$, such that 
	\begin{equation}\label{co9}
	\int_{\mathbb{R}^N}\int_{\mathbb{R}^N}\frac{f(x)h(y)}{|x-y|^{\mu}}~dxdy \leq C(t,r,\mu,N) \|f\|_{L^t} \|h\|_{L^r}.
	\end{equation}
	If $t=r=2N/(2N-\mu)$, then 
	\begin{align*}
	C(t,r,\mu,N)=C(N,\mu)= \pi^{\frac{\mu}{2}}\frac{\Gamma(\frac{N}{2}-\frac{\mu}{2})}{\Gamma(N-\frac{\mu}{2})}\left\lbrace \frac{\Gamma(\frac{N}{2})}{\Gamma(\frac{\mu}{2})}\right\rbrace^{-1+\frac{\mu}{N}}.
	\end{align*}
	Equality holds in  \eqref{co9} if and only if $f\equiv (constant)h$ and 
	\begin{align*}
	h(x)= A(\gamma^2+|x-a|^2)^{(2N-\mu)/2},
	\end{align*} 
	for some $A\in \mathbb{C}, 0\neq \gamma \in \mathbb{R}$ and $a \in \mathbb{R}^N$. \QED
\end{Proposition}
\noi The best constant for the embedding $D^{1,2}(\mathbb{R}^N)$ into  $L^{2^*}(\mathbb{R}^N)$ is defined as 
\[S=\inf_{ u \in D^{1,2}(\mathbb{R}^N) \setminus \{0\}} \left\{ \int_{\mathbb{R}^N} |\na u|^2 dx:\; \int_{\mathbb{R}^N}|u|^{2^{*}}dx=1\right\}.\]
Consequently, we define 
\begin{equation}\label{nh5}
S_{H,L}= \inf_{ u \in D^{1,2}(\mathbb{R}^N)\setminus \{0\}} \left\{  \int_{\mathbb{R}^N}|\nabla u|^2 dx:\;  \int_{\mathbb{R}^{N}} \int_{\mathbb{R}^N } \frac{|u(x)|^{2^{*}_{\mu}}|u(y)|^{2^{*}_{\mu}}}{|x-y|^{\mu}}~dxdy=1 \right\}. 
\end{equation}

\begin{Lemma}\label{khlem7}
	\cite{yang}
	The constant $S_{H,L}$ defined in \eqref{nh5} is achieved if and only if 
	\begin{align*}
	u=C\left(\frac{b}{b^2+|x-a|^2}\right)^{\frac{N-2}{2}}
	\end{align*} 
	where $C>0$ is a fixed constant, $a\in \mathbb{R}^N$ and $b\in (0,\infty)$ are parameters. Moreover,
	\begin{align*} 
S=	S_{H,L} \left[C(N,\mu)\right]^{\frac{N-2}{2N -\mu}}.
	\end{align*}
\end{Lemma}
\begin{Lemma}\cite{yang}
	For $N\geq 3$ and $0<\mu<N$. Then 
	\begin{align*}
	\|.\|_{NL}:= \left(\ds \int_{\mathbb{R}^N}\int_{\mathbb{R}^N}\frac{|.|^{2^{*}_{\mu}}|.|^{2^{*}_{\mu}}}{|x-y|^{\mu}}~dxdy\right)^{\frac{1}{2.2^{*}_{\mu}}}
	\end{align*}
	defines a norm on $L^{2^*}(\mb R^N)$.
\end{Lemma}
The energy functional associated with the problem $(P_\la)$ is $J_\la:H^1_0(\Om) \ra \mathbb{R} $ defined as 
\begin{equation*}
J_\la(u) = \frac{a}{2}\|u\|^2+\frac{\e^p}{2\theta} \|u\|^{2\theta}  - \frac{1}{q} \int_{\Om}f(x)|u|^q ~dx- \frac{1}{2.2^*_{\mu}} \int_{\Om}\int_{\Om} \frac{|u(x)|^{2^{*}_{\mu}}|u(y)|^{2^{*}_{\mu}}}{|x-y|^{\mu}}~dxdy. 
\end{equation*}
 By using Hardy-Littlewood-Sobolev inequality \eqref{co9},  we have 
\begin{align*}
\ds \int_{\Om}\int_{\Om}\frac{|u(x)|^{2^{*}_{\mu}}|u(y)|^{2^{*}_{\mu}}}{|x-y|^{\mu}}~dxdy \leq C(N,\mu)\|u\|_{L^{2^*}}^{2.2^{*}_{\mu}}. 
\end{align*}
It implies the functional $J_\la \in C^1(H_0^{1}(\Om), \mathbb{R})$. Moreover, 
\begin{align*}
\langle J_\la^{\prime}(u), v \rangle  =& (a + \e^p \|u\|^{2(\theta-1)}) \int_{\Om}\na u \cdot \na v~dx - \int_{\Om}f(x)| u|^{q-2}u v~dx\\
& - \int_{\Om}\int_{\Om} \frac{|u(x)|^{2^{*}_{\mu}}|u(y)|^{2^{*}_{\mu}-2}u(y)v(y)}{|x-y|^{\mu}}~dxdy \text{  for all } v \in H^1_0(\Om). 
\end{align*}
\noi  To study the critical points of the problem $(P_\la)$, we consider the   Nehari manifold
\begin{align*}
N_\la:=\{u \in H_0^1(\Om)\setminus\{0\}\; |\;  \ld J_\la^\prime(u),u\rd =0\},
\end{align*}
where   $\ld\;, \;\rd $ denotes  the usual duality.  
Since $ N_\la$ contains every non-zero solution of $(P_\la)$ and we know that the Nehari manifold is closely
related to the behavior of the  fibering maps $\phi_u:\mathbb{R}\rightarrow \mathbb{R}$ as  $\phi_u(t)=  J_{\la}(tu)$,  for $u\in H^1_0(\Om)$. 
It implies $tu\in N_{\la}$ if and only if
$\phi_{u}^{\prime}(t)=0$ and in particular, $u\in N_{\la}$ if
and only if $\phi_{u}^{\prime}(1)=0$. Hence, it is natural to split
$N_{\la}$ into three parts corresponding to the points of local minima,
local maxima and the points of inflection, namely
\begin{align*}
N_{\la}^{0}:= \left\{u\in N_{\la}:\phi_{u}^{\prime\prime}(1) = 0\right\}, N_{\la}^{+}:= \left\{u\in N_{\la}:
\phi_{u}^{\prime\prime}(1)>0\right\},\; N_{\la}^{-}:= \left\{u\in N_{\la}:
\phi_{u}^{\prime\prime}(1)
<0\right\}.
\end{align*}
 \begin{Lemma}\label{khlem2}
 	$J_\la$ is coercive and bounded below on $N_\la$.
 \end{Lemma}

\begin{proof}
	For $u \in N_{\lambda},$  using H\"older's inequality, we have
	\begin{align*}
	J_{\lambda}(u) &= a \left(\frac{1}{2}-\frac{1}{2.2^*_{\mu}}\right)\|u\|^{2}+ \e^p\left(\frac{1}{2\theta}-\frac{1}{2.2^*_{\mu}}\right)\|u\|^{2\theta} -\la\left(\frac{1}{q}-\frac{1}{2.2^*_{\mu}}\right)\int_\Omega f(x)|u|^{q}~dx,\\
	&\geq a\left(\frac{1}{2}-\frac{1}{2.2^*_{\mu}}\right)\|u\|^{2}-\lambda \left(\frac{1}{q}-\frac{1}{2.2^*_{\mu}}\right)\|f\|_{L^r}S^{\frac{-q}{2}} \|u\|^{q}.
	\end{align*}
	Thus  $J_\lambda$ is coercive and bounded below in $N_\lambda$ provided $1<q<2 $.  \QED
\end{proof}
\begin{Lemma}\label{khlem6}
	\begin{enumerate}
		\item [(i)] If $u$ is a local minimum or local maximum  of $J_{\lambda}$ on $N_{\lambda}$  and $u \notin N_{\lambda}^{0}.$ Then $u$ is a critical point for $J_{\lambda},$  and
		\item[(ii)] there exists $\sigma>0$ such that $\|u\|> \sigma$ for all $u \in N_\la^-$,
		\item [(iii)] $N_\la^-$ is a closed set in $H_0^1(\Om)$ topology.
		\end{enumerate}
\end{Lemma}
\begin{proof}
	See \cite{dpp}. \QED
\end{proof}
\begin{Lemma}\label{khlem1}
	There exists $\lambda_{0} > 0$ such that for all $\lambda \in (0, \lambda_{0})$,  we have $N_{\lambda}^{0} = \emptyset.$
\end{Lemma}
\begin{proof}
	We divide the proof into two case:\\
	\textbf{Case 1:}  $u \in N_{\lambda}$ such that $ \ds \int_{\Omega}f(x) |u|^{q}~dx = 0.$\\
	Since $\phi_{u}^{\prime}(1)=0$, we have $a\|u\|^2+\e^p\|u\|^{2\theta}- \|u\|_{NL}^{2.2^*_{\mu}}=0.$ As a result, 
		\begin{equation*}
	\begin{aligned}
\phi_{u}^{\prime\prime}(1)= (2-2.2^*_{\mu} )a \|u\|^{2}-  (2\theta-2.2^*_{\mu} ) \e^p\|u\|^{2\theta} <0.
	\end{aligned}
	\end{equation*}
	\noi which implies  $u \notin N_{\lambda}^{0}$.\\
	\textbf{Case 2:} $u \in N_{\lambda}$ such that  $ \ds \int_{\Omega}f(x) |u|^{q}~dx \neq 0.$\\
	If $u \in N_{\lambda}^{0}$  then, $\phi_u^\prime(1)=0$ and $ \phi_{u}^{\prime\prime}(1)=0$. Therefore, we get
	\begin{equation}\label{kh18}
	\|u\| \geq \left(\frac{(2-q)a S_{H,L}^{2^*_{\mu}}}{2.2^*_{\mu}-q}\right)^{\frac{1}{2.2^*_{\mu}-2}}.
	\end{equation}
	Define $F_{\lambda}:N_{\lambda} \rightarrow \mathbb{R}$ as
	\begin{equation}\label{flambda}
	F_{\lambda}(u) = \frac{(2.2^*_{\mu}-2)a \|u\|^2 + (2.2^*_{\mu}-2\theta)\e^p\|u\|^{2\theta} }{(2.2^*_{\mu}-q)} - \lambda\int_\Omega f(x)|u|^{q}~dx.
	\end{equation}
	Then   $F_{\la}(u) = 0$ for all  $u\; \in N_{\lambda}^{0}.$ Therefore, we get
	\begin{equation*}
	\begin{aligned}
	F_{\lambda}(u)   
	& \geq \|u\|^{q}\left[\left(\frac{2.2^*_{\mu}-2}{2.2^*_{\mu}-q}\right)\|u\|^{(2-q)} -  \la \|f\|_{L^r}S^{\frac{-q}{2}}\right].
	\end{aligned}
	\end{equation*}
	Thus,  using \eqref{kh18}, we obtain
	\begin{equation*}
	F_{\lambda}(u) \geq \|u\|^{q}\left( \left(\frac{2.2^*_{\mu}-2}{2.2^*_{\mu}-q}\right) \left(\frac{(2-q)a S_{H,L}^{2^*_{\mu}}}{2.2^*_{\mu}-q} \right)^{\frac{2-q}{2.2^*_{\mu}-2}} - \la \|f\|_{L^r}S^{\frac{-q}{2}}\right).  \end{equation*}
	Hence, we get
	\begin{equation}\label{kh7}
	0<\lambda_0:=\;\left(\frac{(2.2^*_{\mu}-2)S^{\frac{q}{2}}}{(2.2^*_{\mu}-q)\|f\|_{L^r}}\right)\left(\frac{(2-q)a S_{H,L}^{2^*_{\mu}}}{2.2^*_{\mu}-q} \right)^{\frac{2-q}{2.2^*_{\mu}-2}}
	\end{equation}  
	\noi such that for $ F_{\lambda}(u)>0$ for all $\lambda\in (0, \lambda_0)$  and   $u \in N_{\lambda}^{0},$
	which yields a contradiction. Therefore, $N_{\lambda}^{0} = \emptyset$ for all  $\lambda\in (0, \lambda_0)$.
	\QED
\end{proof}

\noi Now, define $\mathcal{S}_u: \mb R^{+} \lra \mb R$ by
\begin{equation*}
\begin{aligned}
& \mathcal{S}_u(t)= t^{2-q}a\|u\|^{2}+ t^{2\theta-q}\e^p\|u\|^{2\theta}-  t^{2.2^*_{\mu}-q}\|u\|_{NL}^{2.2^*_{\mu}}. 
\end{aligned}
\end{equation*}
\noi Suppose $tu\in N_\la$ then it follows from the definition of $N_\la$ that $\phi_{tu}^{\prime \prime}(1)=t^{q+2}\mathcal{S}_u^{\prime}(t)$  for all  $t>0$. Moreover,  $tu\in N_{\la}$ if and only if $t$ is a
solution of $\mathcal{S}_u(t)={\la}\ds \int_{\Om} f(x) |u|^{q} dx.$
\begin{Lemma}\label{khlem5}
	For each $u\in H_0^1(\Om),\; \la \in (0,\la_0)$ ($\la_0$ is defined in \eqref{kh7}),  the following holds:
	\begin{enumerate}
		\item [(i)]  If $ \int_{\Om} f(x)|u|^{q} dx>0$  then there exists unique $t^+(u),\; t^-(u) >0 $ such that
		\begin{align*}
		t^+(u)<t_{max}<t^-(u),\; t^+(u)u \in N_\la^+ \text{ and }t^-(u)u \in N_\la^-. 
		\end{align*}
		Also,  $\mathcal{S}_u$ is decreasing on $(0, t^+)$, increasing on $(t^+, t^−)$ and decreasing on $(t^−, \infty)$. Moreover,
		\[
	J_\la(t^+u)= \min_{0\leq t\leq t_{max}}J_\la(tu),\quad 	J_\la(t^-u)=\ds \max_{t\geq t^+}J_\la(tu). 
		\]
		\item [(ii)] If $ \int_{\Om} f(x)|u|^{q} dx\leq 0$, then there exists unique $ t^- > t_{max}$ such that $t^-u \in N_\la^-$ and $$J_\la(t^-u)=\ds \sup _{t\geq 0}J_\la(tu).$$
		\item [(iii)] $t^-(u)$ is a continuous function.
		\item [(iv)] $N_\la^-= \ds \bigg\{u \in H_0^1(\Om)\setminus\{0\}\; : \; \frac{1}{\|u\|}t^-\left(\frac{u}{\|u\|}\right)=1\bigg\}$.
	\end{enumerate}
\end{Lemma}
\begin{proof}First we study the behaviour of the function $S_u(t)$ near $0$ and $\infty.$ Taking into account the fact that $1<q<2$ and $2\le 2\theta < 2. 2^*_\mu$, we can choose $t>0$, small enough, such that $S_u(t)>0$  and $\ds\lim_{t\ra \infty}\mathcal{S}_u(t) = -\infty$. Similarly,  $S^\prime_u(t)>0$ for small $t$ and  $\ds\lim_{t\ra \infty}\mathcal{S}_u^\prime(t) = -\infty$. 
	Now we will show that there exists unique $t_{max}>0$ such that $\mathcal{S}_u$ is increasing in $(0,t_{max})$, decreasing in $(t_{max},\infty)$ and $\mathcal{S}_u^{\prime}(t_{max})=0$.
	Set \begin{align*}
	\mathcal{A}_u(t)= (2-q)a\|u\|^{2}+(2\theta-q) \e^pt^{2\theta-2}\|u\|^{2\theta}- (2.2^*_\mu-q) t^{2.2^*_{\mu}-2}\|u\|_{NL}^{2.2^*_{\mu}}.
	\end{align*}
	That is, $\mathcal{A}_u(t)=t^{q-1} \mathcal{S}^\prime_u(t)$.
	So it is enough to show that there exists unique $t_{max}>0$ such that $\mathcal{A}_u(t_{max})=0$.
	We can write $\mathcal{A}_u(t)= (2-q)a\|u\|^{2}- \mathcal{B}_u(t)$, where 
	\begin{align*}
	\mathcal{B}_u(t)= (2.2^*_\mu-q) t^{2.2^*_{\mu}-2}\|u\|_{NL}^{2.2^*_{\mu}}-(2\theta-q) \e^pt^{2\theta-2} \|u\|^{2\theta}.
	\end{align*}
	Since $\theta< 2^*_\mu,\; \mathcal{B}_u(0)=0,\; \mathcal{B}_u(t)<0$ for small $t, \;\mathcal{B}_u(t)>0$ for large $t$ and $\mathcal{B}_u(t)\ra \infty$ as $t \ra \infty$. Moreover there exists a unique $t^*>0$ such that $\mathcal{B}_u(t^*)=0$. Indeed,
	\begin{align*}
	t^*= \left( \frac{(2\theta-q) \e^p\|u\|^{2\theta} }{(2.2^*_\mu-q)\|u\|_{NL}^{2.2^*_{\mu}}}   \right)^{\frac{1}{2.2^*_\mu-2\theta}}.
	\end{align*}
	Hence, there exists unique $t_{max}>t^*>0$ such that 
	$\mathcal{B}_u(t_{max})= (2-q)a\|u\|^{2}$. That is, $\mathcal{A}_u(t_{max})=0$. Thus there exists unique $t_{max}>0$ such that $\mathcal{S}_u$ is increasing in $(0,t_{max})$, decreasing in $(t_{max},\infty)$ and $\mathcal{S}_u^{\prime}(t_{max})=0$. This implies $\phi_{t_{max}u}^{\prime \prime}(1)=0$. Thus,
\begin{align*}
(2-q)t^{2}_{max}a \|u\|^{2}
&\leq (2-q)t^{2}_{max}\|u\|^{2}+  (2\theta-q) \e^pt^{2\theta}_{max}\|u\|^{2\theta}\\&=(2.2^*_{\mu}-q) t^{2.2^*_{\mu}}_{max}\|u\|_{NL}^{2.2^*_{\mu}} 
\\& \leq (2.2^*_{\mu}-q) t^{2.2^*_{\mu}}_{max}S_{H,L}^{-2^*_{\mu}}\|u\|^{2.2^*_{\mu}}.
\end{align*}
Therefore,
\begin{equation}\label{eqb3}
t_{\max}\ge \frac{1}{\|u\|}\left(\frac{(2-q)a S_{H,L}^{2^*_{\mu}}}{2.2^*_{\mu}-q}\right)^{\frac{1}{2.2^*_{\mu}-2}}:= T_1
\end{equation}
Now, since $S_u$ is increasing in $(0,t_{\max})$,  using \eqref{eqb3} we obtain, 
\begin{equation*}
\begin{aligned}
\mathcal{S}_u(t_{max})\geq  \mathcal{S}_u(T_1)&\geq T_1^{2-q}a \|u\|^{2}-T_1^{2.2^*_{\mu}-q}S_{H,L}^{-2^*_{\mu}}\|u\|^{2.2^*_{\mu}}\\
&=\|u\|^{q}a\left(\frac{2.2^*_{\mu}-2}{2.2^*_{\mu}-q}\right)\left(\frac{(2-q)a S_{H,L}^{2^*_{\mu}}}{2.2^*_{\mu}-q}\right)^{\frac{2-q}{2.2^*_{\mu}-2}}> 0.
\end{aligned}
\end{equation*}
{\bf Proof of $(i)$}: 
Since $ \int_{\Om} f(x)|u|^{q} dx>0$,  there exist $0< t^+<t_{max} <t^-$ such that 
\begin{align*}
\mathcal{S}_u(t^+)=\mathcal{S}_u(t^-)=\la \ds \int_{\Om} f(x) |u|^{q} dx.
\end{align*}
 This implies  $t^+u, t^-u \in N_\la.$ Also,  since $\la<\la_0$, $\mathcal{S}_u^{\prime}(t^+)>0$, $\mathcal{S}_u^{\prime}(t^-)<0$ implies $t^+u \in N^{+}_{\lambda}$ and $t^-u \in N^{-}_{\lambda}.$ 
 Indeed,  
$\phi^{\prime}_{u}(t) = t^{q}\bigg(\mathcal{S}_u(t)- \lambda \int_{\Omega} f(x)|u|^{q}~dx\bigg).$
 So $\phi^{\prime}_{u}(t)<0$ for all $t \in [0, t^+)$ and $\phi^{\prime}_{u}(t)>0$ for all $t \in (t^+, t^-)$.  Thus
 \begin{align*}
 J_\la(t^+u) = \ds\min_{0 \leq t \leq t_{max}}J_\la (tu).
 \end{align*}
  In addition, $\phi^{\prime}_{u}(t) > 0$ for all $t \in [t^+, t^-),\;
\phi^{\prime}_{u}(t^-) = 0$ and $\phi^{\prime}_{u}(t) < 0$ for all $t \in (t^-, \infty)$ implies that \[J_\la(t^-u)
= \ds\max_{t \geq t^+} J_\la(tu).\]
{\bf Proof of $(ii):$} 
Similarly, as in the part (i),  we have  $\mathcal{S}_u(t_{max})>0$. As $\lambda\ds \int_{\Om} f(x)|u|^{q}~ dx<0$, it implies there exist unique $t^-$ such that $\mathcal{S}_u(t^-)=\la \ds \int_{\Om} f(x) |u|^{q} dx $ and $\mathcal{S}_u^{\prime}(t^-)<0$ which implies $t^-u\in N_{\la}^{-}$ and the proof of (ii) follows.    For part $(iii)$ and $(iv)$ we refer to Lemma 2.5 of \cite{wu2008}.\QED
\end{proof}
Now let us define $$\ds \theta_\la= \ds   \inf_{u \in N_\la}J_\la(u),\quad \;\theta_\la^+= \ds \inf_{u \in N_\la^+}J_\la(u),\; \text{ and} \;\theta_\la^-= \ds \inf_{u \in N_\la^-}J_\la(u).$$ Then we have
\begin{Lemma}\label{khlem3}
	There exists  $C>0$ such that $\ds\theta_\la^+< - \frac{(2-q)(N-\mu+2)}{2q(2N-\mu)}aC<0$. 
\end{Lemma}
\begin{proof}
	Let $u_0 \in H_0^1(\Om) $ and $ \int_{\Om} f(x)|u_0|^{q} dx>0$  then there exists a unique $t^+>0$ such that $t^+u_0\in N_\la^+$. Hence $\phi_{t^+u_0}^{\prime}(1)=0$ and $\phi_{t^+u_0}^{\prime\prime}(1)>0$. As a result, we get 
	\[
	(2.2^*_{\mu}-q)\|t^+u_0\|_{NL}^{2.2^*_{\mu}}< (2-q)a\|t^+u_0\|^2+(2\theta-q)\e^p\|t^+u_0\|^{2\theta}\] and 
	\[ J_\la(t^+u_0)= \left(\frac{1}{2}-\frac{1}{q}\right)a \|t^+u_0\|^2+ \left(\frac{1}{2\theta}-\frac{1}{q}\right)\e^p\|t^+u_0\|^{2\theta}+ \left(\frac{1}{q}-\frac{1}{2.2^*_{\mu}}\right)\|t^+u_0\|_{NL}^{2.2^*_{\mu}}.\]
	It implies 
	\begin{align*}
	J_\la(t^+u_0)& < - \left(\frac{2-q}{2q}\right)\left(\frac{2^*_{\mu}-1}{2^*_{\mu}}\right)a \|t^+u_0\|^2- \left(\frac{2\theta-q}{2q}\right)\left(\frac{1}{\theta}-\frac{1}{2^*_{\mu}}\right)\e^p\|t^+u_0\|^{2\theta}\\
	& \leq - \frac{(2-q)(N-\mu+2)}{2q(2N-\mu)}a C ,
	\end{align*}
	where $C=\|t^+u_0\|^2 $. Thus, $\theta_\la^+= \ds \inf_{u \in N_\la^+}J_\la(u)\leq J_\la(t^+u_0)< 
	 - \frac{(2-q)(N-\mu+2)}{2q(2N-\mu)}aC<0$.
	\QED
\end{proof}
 \begin{Lemma}\label{khlem4}
 	Let $\lambda \in (0, \lambda_{0}),$ and $u \in N_\la$ then there exists $\de > 0$ and a differentiable function
 $\xi : \mathcal{B}(0,\de) \subseteq H_0^1(\Om) \ra \mathbb{R}^{+}$ such that $\xi(0)=1,$ the function $\xi(v)(u-v)\in N_\la$
 and
 \begin{equation}\label{kh6}
 \langle\xi^{\prime}(0), v\rangle = \frac{\ds (2a + 2\theta\e^p\|u\|^{2(\theta-1)}) \int_{\Om} \na u \cdot \na v ~dx  - q\lambda \int_\Omega f(x)|u|^{q-2}uv\;dx- 2.2^*_{\mu} \mathcal{K}(u,v) }{(2-q)a \|u\|^{2} + (2\theta-q)\e^p \|u\|^{2\theta}-(2.2^*_{\mu}-q)\|u\|_{NL}^{2.2^*_{\mu}} },
 \end{equation}
 for all $u, v\in  H_0^1(\Om)$,  where $\mathcal{K}(u,v) = \ds \int_{\Om} \int_{\Om} \frac{|u(x)|^{2^*_{\mu}}|u(y)|^{2^*_{\mu}-2}u(y)v(y)}{|x-y|^{\mu}} ~dx dy $.
 \end{Lemma}
\begin{proof}
	For  $u\in N_\la$, define  a function $\mathcal{H}_u:\mathbb{R}\times H_0^1(\Om) \ra \mathbb{R}$ given by
	\begin{equation*}
	\begin{aligned}
	\mathcal{H}_u(t,v) &:= \langle J^{\prime}_{\la}(t(u-v)),(t(u-v))\rangle\\& \;=
	t^{2}a\|u-v\|^{2}+\e^p  t^{2\theta}\|u-v\|^{2\theta}-\la t^{q} \int_\Om {f(x)|u-v|^{q}dx}-t^{2.2^*_{\mu}}\|u-v\|_{NL}^{2.2^*_{\mu}}.
	\end{aligned}
	\end{equation*}
	\noi Then $\mathcal{H}_u(1,0) = \langle J^{\prime}_{\la} (u),u \rangle = 0$ which on using Lemma \ref{khlem1} gives 
	\begin{equation*}
	\frac{\partial}{ \partial t}\mathcal{H}_u(1,0)= (2-q)a\|u\|^{2}+(2\theta-q)\e^p\|u\|^{2\theta} -(2.2^*_{\mu}-q)\|u\|_{NL}^{2.2^*_{\mu}} \neq 0. 
	\end{equation*}
	\noi By  Implicit Function Theorem, there exist $\epsilon >0$ and a  differentiable function $\xi : \mathcal{B}(0, \de) \subseteq H_0^1(\Om)  \ra \mathbb{R}$ such that $\xi(0) = 1$, $\mathcal{H}_u(\xi(v),v) = 0$, for all $ v \in \mathcal{B}(0, \de)$ and  equation \eqref{kh6} holds. Moreover, 
	$$
	0=  a\|\xi(v)(u-v)\|^{2}+ \e^p\|\xi(v)(u-v)\|^{2\theta} - \lambda \int_\Omega{f(x)|\xi(v)(u-v)|^{q}dx} - \|\xi(v)(u-v)\|_{NL}^{2.2^*_{\mu}}, $$
	for all $v \in \mathcal{B}(0, \de)$. Thus, $ \xi(v)(u-v) \in N_\la$.
	\QED
\end{proof}

\begin{Proposition}\label{khprop2}
	Let $\{u_n\}$ be a $(PS)_c$ sequence for $J_\la$ with 
	\[
-\infty<	c< c_{\infty}:=  \frac{N-\mu+2}{2(2N-\mu)}(aS_{H,L})^{\frac{2N-\mu}{N-\mu+2}}- \widehat{D} \la^{\frac{2}{2-q}}, 
	\]
	where $\widehat{D}= \left(\frac{(2-q)(2\theta-q)}{4 \theta q}\right)   \left(\frac{ 2\theta-q }{2aS(\theta-1)}\right)^{\frac{q}{2-q}}\|f\|_{L^r}^{\frac{2}{2-q}}$.	Then $\{u_n\}$ contains a convergent subsequence.
\end{Proposition}
\begin{proof}
Let $\{u_n\}$ be a  sequence such that
\begin{align*}
J_\la(u_n)\ra c\; \text{and}\;   J_\la^{\prime}(u_n)\ra 0 \quad  \text{ as } n \ra \infty. 
\end{align*} 
	By standard arguments $\{u_n\}$ is a bounded sequence. Then there exists $u\in H_0^1(\Om)$ such that up to a subsequence $u_n\rp u $ weakly in $H_0^1(\Om)$, $ u_n \ra u$ strongly in $L^\gamma(\Om)$ for all $\gamma \in  [1,2^*)$, $u_n\ra u$ a.e in $\Om$, $\|u_n\|\ra  \al $ as a real sequence and there exists $h \in L^2(\Om)$ such that $|u_n(x)| \leq h(x)$  a.e in $\Om$. Hence we can assume that 
	\begin{align*}
	|\na u_n|^2\rp \omega ,\; |u_n|^{2^*}\ra \tau , \; \left( \int_{\Om}\frac{|u_n(y)|^{2^*_{\mu}}}{|x-y|^{\mu}}~dy \right)|u_n|^{2^*_{\mu}} \rp \nu \text{ in the sense of measure}. 
	\end{align*}
By the second concentration-compactness principle (See \cite{gaoyang}), there exist  at most countable set $I$, 
 sequence of points $\{z_i\}_{i\in I}\subset \mathbb{R}^N$ and  families of positive numbers $\{v_i:i\in I\}$, $\{w_i:i\in I\}$ and $\{x_i:i\in I\}$  such that
 \begin{align*}
 & \nu= \left( \int_{\Om}\frac{|u(y)|^{2^*_{\mu}}}{|x-y|^{\mu}}~dy \right)|u|^{2^*_{\mu}} + \sum_{i \in I } v_i \de_{z_i},\\
 & \omega\geq |\na u|^2+\sum_{i \in I } w_i \de_{z_i} ,\; \tau \geq |u|^{2^*} + \sum_{i \in I } x_i \de_{z_i}, \\
 & \quad S_{H,L} v_i^{\frac{1}{2^*_{\mu}}}\leq w_i,   \text{ and } \; v_i\leq C(N,\mu) x_i^{\frac{2N-\mu}{N}}, 
 \end{align*}
where $\de_{z_i}$ is the Dirac mass at $z_i$. 
Moreover, we can construct a smooth cut-off function $\varphi_{\e,i}$ centered at $z_i$ such that 
\begin{align*}
0\leq \varphi_{\e,i}(x)\leq 1,\; \varphi_{\e,i}(x)= 1  \text{ in } B\left(z_i, \frac{\e}{2}\right) ,\; \varphi_{\e,i}(x)= 0  \text{ in } \mathbb{R}^N \setminus B\left(z_i, \e \right),\; |\na\varphi_{\e,i}(x) |\leq \frac{4}{\e},
\end{align*}
	for any $\e>0$ small.  Observe that 
	\begin{align*}
	\left |\int_{\Om} f(x) |u_n|^q\varphi_{\e,i}~dx \right|& 
	= 	\left |\int_{B\left(z_i, \e \right)} f(x) |u_n|^q\varphi_{\e,i}~dx \right| \\
	&\leq  \|f\|_{L^r} \left(\int_{B\left(z_i, \e \right)}  |u_n|^{2^*}~dx\right)^{\frac{q}{2^*}} \ra 0 \text{ as } \e \ra 0.
	\end{align*}
	It implies  $\ds \lim_{\e\ra 0}  \lim_{ n \ra \infty} \int_{\Om} f(x) |u_n|^q\varphi_{\e,i}~dx = 0$.
	 Therefore,\\
	\begin{align*}
0&=  \ds \lim_{\e\ra 0}  \lim_{ n \ra \infty} \ld J_\la^{\prime} (u_n), \varphi_{\e,i}u_n \rd\\
& = \ds \lim_{\e\ra 0}  \lim_{ n \ra \infty}  \bigg\{ \left( a+\e^p \|u_n\|^{2(\theta-1)}\right) \int_{\Om} \na u_n\cdot \na (\varphi_{\e,i}u_n)~dx - \int_{\Om} |u_n|^q \varphi_{\e,i} ~ dx\\
& \hspace{3cm} - \int_{\Om} \int_{\Om} \frac{|u_n(x)|^{2^*_{\mu}}|u_n(y)|^{2^*_{\mu}}\varphi_{\e,i}(y)}{|x-y|^{\mu}} ~dx dy \bigg\}\\
& \geq \ds \lim_{\e\ra 0}  \lim_{ n \ra \infty}  \bigg\{ \left( a+\e^p \|u_n\|^{2(\theta-1)}\right) \int_{\Om} \na u_n\cdot \na (\varphi_{\e,i}u_n)~dx \\
&\hspace{3cm}- \int_{\Om} \int_{\Om} \frac{|u_n(x)|^{2^*_{\mu}}|u_n(y)|^{2^*_{\mu}}\varphi_{\e,i}(y)}{|x-y|^{\mu}} ~dx dy \bigg\}\\
& \geq \ds \lim_{\e\ra 0}  \lim_{ n \ra \infty}  \bigg\{ \left( a+\e^p \|u_n\|^{2(\theta-1)}\right) \int_{\Om} \left( |\na u_n|^2 \varphi_{\e,i} +  u_n \na u_n\cdot \na \varphi_{\e,i}\right) ~dx\\
& \hspace{3cm} - \int_{\Om} \int_{\Om} \frac{|u_n(x)|^{2^*_{\mu}}|u_n(y)|^{2^*_{\mu}}\varphi_{\e,i}(y)}{|x-y|^{\mu}} ~dx dy \bigg\}
\end{align*} \begin{align*}
& \geq \ds \lim_{\e\ra 0}  \lim_{ n \ra \infty}  \bigg\{  a \int_{\Om}  |\na u_n|^2 \varphi_{\e,i}  ~dx - \int_{\Om} \int_{\Om} \frac{|u_n(x)|^{2^*_{\mu}}|u_n(y)|^{2^*_{\mu}}\varphi_{\e,i}(y)}{|x-y|^{\mu}} ~dx dy \bigg\}\\
& \geq \ds \lim_{\e\ra 0}  \lim_{ n \ra \infty}  \bigg\{  a \int_{\Om}   \varphi_{\e,i}  ~d\omega - \int_{\Om}  \varphi_{\e,i} ~d\nu \bigg\}\\&
 \geq a w_i- v_i. 
	\end{align*}
	Therefore, $a w_i\leq  v_i$. Combining this with the fact that  $ S_{H,L} v_i^{\frac{1}{2^*_{\mu}}}\leq w_i$, we obtain 
	\begin{align}\label{kh22}
w_i\geq \left(a S_{H,L}^{2^*_{\mu}}\right)^{\frac{1}{2^*_{\mu}-1}} \text{ or } w_i= 0. 
	\end{align}
	Using H\"older's inequality, Sobolev embedding and Young's inequality, we get 
	\begin{equation}
\begin{aligned}\label{kh1}
\la \int_{\Om}f(x)|u|^q &~dx  \leq \la \|f\|_{L^r} S^{-\frac{q}{2}}\|u\|^q\\
& = \left(\left[\frac{a(\theta-1)}{\theta q }  \left[\frac{1}{q}-\frac{1}{2\theta}\right]^{-1} \right]^{\frac{q}{2}} \|u\|^q \right)\left(\left[\frac{a(\theta-1)}{\theta q }  \left[\frac{1}{q}-\frac{1}{2\theta}\right]^{-1} \right]^{\frac{-q}{2}}  \la \|f\|_{L^r} S^{-\frac{q}{2}} \right)\\
& \leq \frac{a(\theta-1)}{2\theta  }  \left[\frac{1}{q}-\frac{1}{2\theta}\right]^{-1}  \|u\|^2+  \la^{\frac{2}{2-q}}\frac{2-q}{2} \left(\frac{ 2\theta-q }{2aS(\theta-1)}\right)^{\frac{q}{2-q}}\|f\|_{L^r}^{\frac{2}{2-q}}. 
\end{aligned}
\end{equation}
We claim that the set $I$ is empty. Suppose not, that is, there exists $i_0\in I$ such that  $w_{i_0}\geq \left(a S_{H,L}^{2^*_{\mu}}\right)^{\frac{1}{2^*_{\mu}-1}}$. Then using \eqref{kh1},  we have 
\begin{align*}
c=&  \lim_{n\ra \infty}J_\la(u_n)- \frac{1}{2\theta}\ld J_\la^{\prime}(u_n), u_n \rd \nonumber\\
& =  \lim_{n\ra \infty} \bigg\{ \frac{a(\theta-1)}{2\theta  } \|u_n\|^2 -\la \left(\frac{1}{q}-\frac{1}{2\theta}\right) \int_{\Om}|u_n|^q~dx + \left(\frac{1}{2\theta}-\frac{1}{2.2^*_{\mu}}\right) \|u_n\|_{NL}^{2.2^*_{\mu}} \bigg\}\nonumber \\
& \geq \bigg\{ \frac{a(\theta-1)}{2\theta  } \left( \|u\|^2 +\sum_{i \in I } w_i \right)  -\la \left(\frac{1}{q}-\frac{1}{2\theta}\right) \int_{\Om}f(x)|u|^q~dx\nonumber \\
& \hspace{2cm} + \left(\frac{1}{2\theta}-\frac{1}{2.2^*_{\mu}}\right) \left( \|u\|_{NL}^{2.2^*_{\mu}} +  \sum_{i \in I } v_i \right) \bigg\}\\
& \geq \bigg\{ \frac{a(\theta-1)}{2\theta  }  w_{i_0}   - \la^{\frac{2}{2-q}} \left(\frac{(2-q)(2\theta-q)}{4 \theta q}\right)   \left(\frac{ 2\theta-q }{2aS(\theta-1)}\right)^{\frac{q}{2-q}}\|f\|_{L^r}^{\frac{2}{2-q}} \bigg\}\\
& \geq    \left(\frac{1}{2}-\frac{1}{2.2^*_{\mu}}\right) a w_{i_0} - \widehat{D} \la^{\frac{2}{2-q}} \geq    \frac{N-\mu+2}{2(2N-\mu)} \left(a S_{H,L}\right)^{\frac{2^*_{\mu}}{2^*_{\mu}-1}} - \widehat{D} \la^{\frac{2}{2-q}} \nonumber.
\end{align*}
	 This yields a contradiction. Thus $I$ is empty and 
	\begin{align*}
	\int_{\Om} \int_{\Om} \frac{|u_n(x)|^{2^*_{\mu}}|u_n(y)|^{2^*_{\mu}}}{|x-y|^{\mu}} ~dx dy \ra \int_{\Om} \int_{\Om} \frac{|u(x)|^{2^*_{\mu}}|u(y)|^{2^*_{\mu}}}{|x-y|^{\mu}} ~dx dy  \text{ as } n \ra \infty. 
	\end{align*}
	Now using the fact that $ \ld J_\la^{\prime}(u_n), u_n \rd \ra 0 $ and $\ld J_\la^{\prime}(u_n), u \rd \ra 0$  we have
	\begin{align*}
	& (a+\e^p\al^{2(\theta-1)})\al^2= \la \int_{\Om}|u|^q dx+ \int_{\Om} \int_{\Om} \frac{|u(x)|^{2^*_{\mu}}|u(y)|^{2^*_{\mu}}}{|x-y|^{\mu}} ~dx dy +o_n(1),\\& (a+\e^p\al^{2(\theta-1)})\|u\|^2= \la \int_{\Om}|u|^q dx+ \int_{\Om} \int_{\Om} \frac{|u(x)|^{2^*_{\mu}}|u(y)|^{2^*_{\mu}}}{|x-y|^{\mu}} ~dx dy +o_n(1). 
	\end{align*}
	As a result, we get $\|u_n\|^2 \ra \al^2= \|u\|^2$. Hence the proof follows. \QED
\end{proof}
	\begin{Proposition}\label{khprop1}
		Let $\lambda \in (0,\lambda_{0})$,  then there exists a  sequence $\{u_n\} \subset N_\la$ such that
		\begin{align*}
			J_\la(u_{n}) = \theta_{\lambda}+o_n(1) \text{ and } J_\la^{\prime}(u_{n}) = o_n(1).
		\end{align*}
	\end{Proposition}
	\begin{proof}
		Using Lemma \ref{khlem2} and  Ekeland variational principle \cite{eke1974}, there exists a minimizing sequence $\{u_n\}\subset N_\la $ such that
		\begin{equation}\label{kh2}
		 J_\la(u_n) < \theta_\la +\frac{1}{n} \text{ and }
		 		J_\la(u_n)< J_\la(v)+\frac{1}{n}\|v-u_n\| \; \textrm{for all}\; v \in N_\la.
		\end{equation}
		For large $n$, using equation \eqref{kh2} and Lemma \ref{khlem3}, we have 
		\begin{align*}
		J_\la (u_n) < \theta_\la +\frac{1}{n}<\theta_\la^{+}<0.
		\end{align*}
		From the fact that $J_\la (u_n) <\theta_\la^{+}<0 $ and using  H\"older's inequality, Sobolev embedding, 
		\begin{align}\label{kh3}
		\left(\frac{\;(-\theta_\la^{+})2.2^*_{\mu}q S^{\frac{q}{2}}}{ (2.2^*_{\mu}-q) \la \|f\|_r} \right)^{\frac{1}{q}}\leq \|u_n\|\leq  \left(\frac{\;\la(2.2^*_{\mu}-q)\|f\|_r }{qa  (2^*_{\mu}-1)S^{\frac{q}{2}}}\right)^{\frac{1}{2-q}}.
			\end{align}		
		\noi Now, we prove that $\|J^{\prime}_\lambda(u_n)\|\rightarrow 0$ as $n\rightarrow \infty$. Applying Lemma \ref{khlem4}, for $u_n$ we obtain differentiable functions $\xi_n:\mathcal{B}(0, \de_n)\rightarrow \mathbb{R}$ for some $\de_n>0$ such that $\xi_n(v)(u_n-v)\in N_\la $,\; $\textrm{for all}\; v\in \mathcal{B}(0, \de_n).$
		\noindent Fix $n$, choose $0<\rho<\de_n$. Let $u\in H_0^1(\Om)$ with $u\not\equiv 0$ and let $v_\rho=\frac{\rho u}{\|u\|}$. We set $h_\rho=\xi_n(v_\rho)(u_n-v_\rho) \in N_\lambda$. Using \eqref{kh2}, we have 
		\begin{align*}
		J_\la(h_\rho)-J_\la(u_n)\geq-\frac{1}{n}\|h_\rho-u_n\|.
		\end{align*}
		Applying Mean Value Theorem , we get 
		\begin{align*}
		\langle J^{\prime}_{\la}(u_n),h_\rho-u_n \rangle+ o_n(\|h_\rho-u_n\|)\geq -\frac{1}{n}\|h_\rho-u_n\|.
		\end{align*}
		Thus,
		\begin{align*}
		\bigg\langle J^{\prime}_{\la}(u_n),\frac{u}{\|u\|} \bigg\rd \leq  \frac{\|h_\rho-u_n\|}{n \rho}+\frac{o_n(\|h_\rho-u_n\|)}{\rho}+\frac{(\xi_n(v\rho)-1)}{\rho}\langle J^{\prime}_{\la}(u_n)-J^{\prime}_{\la}(h\rho),u_n-v\rho \rangle. 
		\end{align*}
		\noi Since  $\displaystyle\lim_{n\rightarrow \infty}\frac{|\xi_n(v_\rho)-1|}{\rho}\leq \|\xi_n(0)\|$ and $\|h_\rho-u_n\|\leq \rho|\xi_n(v_\rho)|+|\xi_n(v_\rho)-1|\; \|u_n\|$.\\
		Therefore, taking  $\rho \ra 0$, we can find a constant $C>0$ independent of $\rho$, such that 
		\begin{equation*}
		\langle J^{\prime}_\lambda(u_n),\frac{u}{\|u\|}\rangle\leq\frac{C}{n}(1+\|\xi_n^{\prime}(0)\|).
		\end{equation*}
		Thus, if we can show that $\|\xi_n^{\prime}(0)\|$ is bounded then we are done. Now, using \eqref{kh6}, \eqref{kh3}, we can show that for some $K>0$, 
		\begin{equation*}
		|\langle \xi_n^{\prime}(0), v\rangle|\leq \frac{K\|v\|}{\bigg|(2-q)a \|u_n\|^{2} + (2\theta-q)\e^p \|u_n\|^{2\theta}-(2.2^*_{\mu}-q)\|u_n\|_{NL}^{2.2^*_{\mu}} \bigg|}.
		\end{equation*}
  Let if possible, there exists a subsequence $\{u_n\}$ of $\{u_n\}$(we still denote it by $\{u_n\}$) such that
		\begin{equation}\label{kh17}
		(2-q)a \|u_n\|^{2} + (2\theta-q)\e^p \|u_n\|^{2\theta}-(2.2^*_{\mu}-q)\|u_n\|_{NL}^{2.2^*_{\mu}} =o_n(1).
		\end{equation}
		From equation \eqref{kh17} and the fact that $u_n \in N_\la$, we get $F_\la(u_n)=o_n(1)$ ($F_\la$ defined in \eqref{flambda}) and
		\begin{equation*}
		\|u_n\| \geq   \left(\frac{(2-q)a S_{H,L}^{2^*_{\mu}}}{2.2^*_{\mu}-q}\right)^{\frac{1}{2.2^*_{\mu}-2}} +o_n(1).
		\end{equation*}
		Now analogous to  the proof of Lemma \ref{khlem1}, we get $F_\lambda(u_n)>0$ for large $n$, which is a contradiction. Hence $\{u_n\}$ is a Palais-Smale sequence for $J_\la$  at the level $\theta_\la$. 
	\QED
	\end{proof}
\begin{Remark}\label{khrem1}
	We remark that by following the proof of Proposition \ref{khprop1},  we can prove that if $\lambda \in (0,\lambda_{0})$, then there exists a  sequence $\{u_n\} \subset N_\la^- $ such that
	\begin{align*}
	J_\la(u_{n}) = \theta_{\lambda}^-+o_n(1) \text{ and } J_\la^{\prime}(u_{n}) = o_n(1).
	\end{align*}
\end{Remark}
	\section{Existence of First solution }
Choose $\la_1>0$  such that 
\begin{align*}
 \la^{\frac{2}{2-q}}\left(\frac{(2-q)(2\theta-q)}{4 \theta q}\right)   \left(\frac{ 2\theta-q }{2aS(\theta-1)}\right)^{\frac{q}{2-q}}\|f\|_{L^r}^{\frac{2}{2-q}}< \frac{N-\mu+2}{2(2N-\mu)}(aS_{H,L})^{\frac{2N-\mu}{N-\mu+2}},
\end{align*}
 whenever $0<\la<\la_1$.	Define 
	\begin{align*}
	\La^* = \min\{ \la_0,\la_1\}.
	\end{align*}
\textbf{Proof of Theorem \ref{khthm1}:}
	From Proposition \ref{khprop1}, there exists a minimizing sequence $\{u_n\} \subset N_\la$ such that
	\begin{align*}
\lim_{n\ra \infty}J_\la(u_{n}) = \theta_{\lambda}\leq \theta_{\lambda}^+<0  \text{ and } \lim_{n\ra \infty} J_\la^{\prime}(u_{n}) = 0.
	\end{align*}
By the choice of $\La^*$, we have
\begin{align*}
 c_{\infty}= \frac{N-\mu+2}{2(2N-\mu)}(aS_{H,L})^{\frac{2N-\mu}{N-\mu+2}}- \la^{\frac{2}{2-q}}\left(\frac{(2-q)(2\theta-q)}{4 \theta q}\right)   \left(\frac{ 2\theta-q }{2aS(\theta-1)}\right)^{\frac{q}{2-q}}\|f\|_r^{\frac{2}{2-q}}>0, 
\end{align*}  
	for all $0<\la<\La^*$. Therefore, $ \theta_\la<0< c_{\infty}$, this on using Proposition \ref{khprop2} gives us that  $\{u_n\}$ contains a convergent subsequence. That is, there exists  $u_1\in H_0^1(\Om)$ such that  $u_n \ra u_1 $ in $H_0^1(\Om)$. It implies that 
	\begin{align*}
	\lim_{n\ra \infty}J_\la(u_{n}) = \theta_{\lambda}= J_\la(u_1). 
	\end{align*}
	Hence $u_1$ is minimizer of $J_\la$  and $u_1 \in N_\la$ for $\la \in (0,\La^*)$. Also, $J_\la(u_1)<0$. Now we claim that $u_1
	\in N_\la^+$. On the contrary, let us assume that $u_1 \in N_\la^-$ then  from Lemma \ref{khlem5}, there exists  $t^+<t^-=1$ such that $t^+u_1 \in  N_\la^+$. Hence $\phi_{u_1}^{\prime}(t^+)=0,\; \phi_{u_1}^{\prime\prime}(t^+)>0$, so $t^+$ is local minimum of $\phi_{u_1}$. Therefore, there exists a $t^*\in (t^+,1)$ such that $J_\la(t^+u_1)<J_\la(t^*u_1 )$. Thus 
	\begin{align*}
	\theta_\la \leq J_\la(t^+u_1)<J_\la(t^*u_1 )\leq J_\la(u_1 ) = \theta_\la,
	\end{align*}
	which is not possible. Thus $u_1\in N_\la^+$ and $\theta_{\lambda}= \theta_\la^+=  J_\la(u_1)$. By using the same  arguments as in \cite[pp.281]{taren}, we get that  $u_1$ is a local minimum for $J_\la$.  Since $J_\la(u_1)=J_\la(|u_1|)$,  by Lemma  \ref{khlem6}, $u_1$ is non-negative solution of $(P_\la)$. Using \cite[Lemma 4.4]{yangjmaa}, we have $u_1\in L^{\infty}(\Om)$ and $u_1 \in C^2(\Om)$. Applying strong maximum principle we get that $u_1>0$ in $\Om$. 
	\QED

\section{Second Solution of $(P_\la)$}
\noi To prove the existence of second solution, we will show that the minimizer of the functional over $N_\la^-$ is achieved and forms the second solution. For this we  use the minimizers of the best constant $S_{H,L}$. From Lemma \ref{khlem7} we know that 
\begin{align*}
U_\e(x)= S^{\frac{(N-\mu)(2-N)}{4(N-\mu+2)}}(C(N,\mu))^{\frac{2-N}{2(N-\mu+2)}}\left(\frac{\e}{\e^2+|x|^2}\right)^{\frac{N-2}{2}}, \; 0<\e<1
\end{align*}
are the minimizers of $S_{H,L}$. Since, $f$ is a continuous function on $\Om$ and $f^+=\max\{f(x),0\}\not\equiv 0$, the set $\Sigma=\{ x \in \Om : f(x)>0 \}$ is an open set of positive measure. Without loss of generality, let us assume that $\Sigma$ is a domain and $0 \in \Sigma$.
 This implies there exists a $\de>0$ such that $B_{4\de}(0)\subset \Sigma \subseteq \Om$ and  $f(x)>0$ for all $x \in B_{2\de}(0)$.
 It implies that there exists a $m_f>0$ such that $f(x)>m_f$ for all $x \in B_{2\de}(0)$. Now  define $\eta\in C_c^{\infty}(\mathbb{R}^N)$ such that $0\leq \eta\leq 1$ in $\mathbb{R}^N$, $\eta\equiv1$ in $B_\de(0)$ and $\eta\equiv 0 $ in $\mathbb{R}^N \setminus B_{2\de}(0)$ and $|\na \eta |< C$. Let $u_\e\in  H_0^1(\Om)$ be defined as  $u_\e(x)= \eta(x) U_\e(x)$. Then we have the following:
\begin{Proposition}\label{khprop3}
	Let $N\geq 3,\; 0<\mu<N$ then the following holds:
	\begin{enumerate}
		\item [(i)] $\|u_\e\|^2 \leq  S_{H,L}^{\frac{2N-\mu}{N-\mu+2}}+O(\e^{N-2})$.
		\item [(ii)] $\|u_\e\|_{NL}^{2.2^*_{\mu}}\leq S_{H,L}^{\frac{2N-\mu}{N-\mu+2}}+O(\e^N)$.
		\item [(iii)] $\|u_\e\|_{NL}^{2.2^*_{\mu}}\geq S_{H,L}^{\frac{2N-\mu}{N-\mu+2}}-O(\e^N)$.
	\end{enumerate}
\end{Proposition}
\begin{proof}
	For part $(i)$ See \cite[Lemma 1.46]{willem}. For $(ii)$ and $(iii)$ See \cite[Proposition 2.8]{systemchoq}.\QED
\end{proof}
\begin{Lemma}\label{khlem9}
	Let $\mu<\min\{4, N\}$ then there exists $\Upsilon^*>0$ and $\e^*>0$  such that for every  $\la \in (0,\Upsilon^*)$ and $\e \in (0,\e^*)$, we have 
	\begin{align*}
	\sup_{t\geq 0}J_\la(u_1+t u_\e)< c_{\infty},
	\end{align*}
	where $u_1$ is the local minimum of $J_\la$ obtained  in Theorem \ref{khthm1} \and $c_\infty$ is defined as in the Proposition \ref{khprop2}.
\end{Lemma}
\begin{proof} 
	From the definition, $u_\e(x)\geq 0$ for all $x \in \mathbb{R}^N$. Let $0<\e<\de$ then $u_\e= U_\e$ in $B_\e(0)$.\\ 
	\textbf{ claim 1:} There exists a $r_1>0$ such that 
	\begin{align*}
	I= \int_{\Om}\int_{\Om}\frac{(u_\e(x))^{2^{*}_{\mu}}(u_\e(y))^{2^{*}_{\mu}-1}u_1(y)}{|x-y|^{\mu}}~dxdy\geq r_1\e^{\frac{N-2}{2}}.
	\end{align*} 
Actually,
	\begin{align*}
	I& \geq \ds \int_{B_\e(0)}\int_{B_\e(0)}\frac{(u_\e(x))^{2^{*}_{\mu}}(u_\e(y))^{2^{*}_{\mu}-1}u_1(y)}{|x-y|^{\mu}}~dxdy\\
	& \geq C  \ds \int_{B_\e(0)}\int_{B_\e(0) }\frac{(U_\e(x))^{2^{*}_{\mu}}(U_\e(y))^{2^{*}_{\mu}-1}}{|x-y|^{\mu}}~dxdy\\
	&\geq  C \ds \int_{B_\e(0)}\int_{B_\e(0)}\frac{\e^{\frac{3N}{2}+1-\mu}~dxdy}{|x-y|^{\mu} (\e^2+|x|^2)^{\frac{2N-\mu}{2}}(\e^2+|y|^2)^{\frac{N-\mu+2}{2}}}\\
	&\geq  C \e^{\frac{N-2}{2}}\ds \int_{B_1(0)}\int_{B_1(0)}\frac{dxdy}{|x-y|^{\mu} (1+|x|^2)^{\frac{2N-\mu}{2}}(1+|y|^2)^{\frac{N-\mu+2}{2}}}= O(\e^{\frac{N-2}{2}}).
	\end{align*}
	This proves the claim 1.  To get the estimate of  $\|u_1+t u_\e\|_{NL}^{2.2^*_{\mu}}$, we divide the proof into  two cases:\\
	\textbf{Case 1:} $ 2^{*}_{\mu} > 3$.\\
	It is easy to see that  there exists $\widehat{A}>0$ such that 
	\begin{align*}
	(a+b)^p\geq a^p+b^p+pa^{p-1}b+\widehat{A}ab^{p-1} \text{ for all } a,\;b \geq 0 \text{ and } p>3, 
	\end{align*}
	which  implies that 
	\begin{align*}
	\|u_1+t u_\e\|_{NL}^{2.2^*_{\mu}}	& \geq \|u_1\|_{NL}^{2.2^*_{\mu}}+\|t u_\e\|_{NL}^{2.2^*_{\mu}} + \widehat{C} t^{2.2^*_{\mu}-1} \int_{\Om}\int_{ \Om}\frac{(u_\e(x))^{2^*_{\mu}}(u_\e(y))^{2^*_{\mu}-1}u_1(y)}{|x-y|^{\mu}}~dxdy\\& \quad + 2.2^*_{\mu}t  \int_{\Om}\int_{ \Om}\frac{(u_1(x))^{2^*_{\mu}}(u_1(y))^{2^*_{\mu}-1}u_\e(y)}{|x-y|^{\mu}}~dxdy, \text{ where } \widehat{C}= \min\{\widehat{A}, 2.2^*_\mu\}.
	\end{align*}
	\textbf{Case 2:} $2< 2^{*}_{\mu} \leq 3$.\\
	In this case, we claim that 
	\begin{align}\label{nh22}
	\|u_1+t u_\e\|_{NL}^{2.2^*_{\mu}}	& 
	\geq \|u_1\|_{NL}^{2.2^*_{\mu}}+\|t u_\e\|_{NL}^{2.2^*_{\mu}} +\widehat{C} t^{2.2^*_{\mu}-1} \int_{\Om}\int_{ \Om}\frac{(u_\e(x))^{2^*_{\mu}}(u_\e(y))^{2^*_{\mu}-1}u_1(y)}{|x-y|^{\mu}}~dxdy\\& \quad + 2.2^*_{\mu} t \int_{\Om}\int_{ \Om}\frac{(u_1(x))^{2^*_{\mu}}(u_1(y))^{2^*_{\mu}-1}u_\e(y)}{|x-y|^{\mu}}~dxdy- O(\e^{(\frac{2N-\mu}{4}) \Theta}) , \nonumber
	\end{align}  for all $ \Theta\in (0,1)$.\\
	\noi We recall the inequality from Lemma 4 of \cite{brezis1}: there exist $ C$(depending on $2^*_{\mu}$) such that, for all $a,b\geq 0$, 
	\begin{align}\label{nh21}
	(a+b)^{2^*_{\mu}}\geq   \left\{
	\begin{array}{ll}
	a^{2^*_{\mu}}+b^{2^*_{\mu}}+2^*_{\mu}a^{2^*_{\mu}-1}b+ 2^*_{\mu}ab^{2^*_{\mu}-1} -C ab^{2^*_{\mu}-1} &  \text{ if } a\geq b , \\
	a^{2^*_{\mu}}+b^{2^*_{\mu}}+2^*_{\mu}a^{2^*_{\mu}-1}b+ 2^*_{\mu}ab^{2^*_{\mu}-1} -C a^{2^*_{\mu}-1}b &  \text{ if } a\leq b, \\
	\end{array} 
	\right. \end{align}	
	Consider $\Om\times \Om= O_1 \cup O_2 \cup O_3 \cup O_4$, where
	\begin{align*}
	& O_1= \{(x,y) \in \Om\times \Om  \mid u_1(x)\geq tu_\e(x) \text{ and } u_1(y)\geq tu_\e(y) \},\\
	& O_2= \{(x,y) \in \Om\times \Om  \mid u_1(x)\geq tu_\e(x) \text{ and } u_1(y)< tu_\e(y) \},\\
	& O_3= \{(x,y) \in \Om\times \Om  \mid u_1(x)< tu_\e(x) \text{ and } u_1(y)\geq tu_\e(y) \}, \\
	& O_4= \{(x,y) \in \Om\times \Om  \mid u_1(x)< tu_\e(x) \text{ and } u_1(y)< tu_\e(y) \}.
	\end{align*} 
	
	Also, define  $b(u)|_{O_i}= \ds \int\int_{O_i}\frac{(u(x))^{2^{*}_{\mu}}(u(y))^{2^{*}_{\mu}}}{|x-y|^{\mu}}~dxdy $,  for all $u \in H_0^1(\Om)$ and $i=1,2,3,4$.
	\textbf{Subcase 1:} when $(x,y) \in O_1$. \\
	From \eqref{nh21}, we have 
	\begin{align*}
b(u_1+t u_\e)_{|O_1} & \geq  b(u_1)_{|O_1} +b(t u_\e)_{|O_1} +2.2^*_{\mu} t^{2.2^*_{\mu}-1} \int\int_{O_1}\frac{(u_\e(x))^{2^*_{\mu}}(u_\e(y))^{2^*_{\mu}-1}u_1(y)}{|x-y|^{\mu}}dxdy\\
	& \quad + 2.2^*_{\mu} t \iint_{O_1} \frac{(u_1(x))^{2^*_{\mu}}(u_1(y))^{2^*_{\mu}-1}u_\e(y)}{|x-y|^{\mu}}~dxdy - A^1_\e,
	\end{align*}
	where $A^1_\e$ is sum of eight non-negative integrals and  each integral  has an  upper bound of the form  $ C \int\int_{O_1}\frac{u_1(x)(tu_\e(x))^{2^*_{\mu}-1} (u_1(y))^{2^*_{\mu}}}{|x-y|^{\mu}}~dxdy$ or $ C \int\int_{O_1}\frac{u_1(y)(tu_\e(y))^{2^*_{\mu}-1} (u_1(x))^{2^*_{\mu}}}{|x-y|^{\mu}}~dxdy$.
		 \begin{align*}
		 A^1_\e \leq C \int\int_{O_1}\frac{u_1(x)(u_\e(x))^{2^*_{\mu}-1} (u_1(y))^{2^*_{\mu}}}{|x-y|^{\mu}}~dxdy +C \int\int_{O_1}\frac{u_1(x)(u_\e(x))^{2^*_{\mu}-1} (u_\e(y))^{2^*_{\mu}}}{|x-y|^{\mu}}~dxdy 
		 \end{align*}
	Now write $(tu_\e(x))^{2^*_{\mu}-1}= (tu_\e(x))^{r}. (tu_\e(x))^{s}$ with $2^*_{\mu}-1= r+s$  and  $0<s<\frac{2^*_{\mu}}{2}$ then 
	utilizing the definition of $O_1$ and  the fact that  $u_1 \in L^{\infty}(\Om)$,  we have  	
	\begin{equation*}
	\begin{aligned}
	\ds \int\int_{O_1}   \frac{u_1(x)(tu_\e(x))^{2^*_{\mu}-1} (u_1(y))^{2^*_{\mu}}}{|x-y|^{\mu}}~dxdy  &  \leq C \int\int_{O_1}\frac{(u_1(x))^{1+r}(tu_\e(x))^{s} (u_1(y))^{2^*_{\mu}}}{|x-y|^{\mu}}~dxdy  \\
	&\leq  C \int_{\Om}\int_{\Om}\frac{(tu_\e(x))^{s} }{|x-y|^{\mu}}~dxdy \\
	&\leq  C \int_{\Om}\int_{\Om}\frac{  \e^{\frac{s(N-2)}{2}}   }{|x-y|^{\mu} |x|^{s(N-2)}}~dxdy \\
	& \leq C \e^{\frac{s(N-2)}{2}} \left( \ds \int_{\Om}
	\frac{dx}{|x|^{\frac{s(2N)(N-2)}{2N-\mu}}}
	\right)^{\frac{2N-\mu}{2N}} \\
	& \leq C \e^{\frac{s(N-2)}{2}} \left( \ds \int_{\Om}
	\frac{dx}{|x|^{\frac{s(2N)(N-2)}{2N-\mu}}}
	\right)^{\frac{2N-\mu}{2N}}.
	\end{aligned}
	\end{equation*}
	By the choice of $s$, we have  $\ds \int_{\Om}
	\frac{dx}{|x|^{\frac{s(2N)(N-2)}{2N-\mu}}} < \infty$. As a result, we get
	\begin{align*}
	\ds \int\int_{O_1}\frac{u_1(x)(tu_\e(x))^{2^*_{\mu}-1} (u_1(y))^{2^*_{\mu}}}{|x-y|^{\mu}}~dxdy \leq O(\e^{(\frac{2N-\mu}{4}) \Theta}) \text{ for all } \Theta \in(0,1).
	\end{align*}	 
	In a similar manner, we have  
	\begin{align*}
	\ds C \int\int_{O_1}\frac{u_1(y)(tu_\e(y))^{2^*_{\mu}-1} (u_1(x))^{2^*_{\mu}}}{|x-y|^{\mu}}~dxdy \leq O(\e^{(\frac{2N-\mu}{4}) \Theta}) \text{ for all } \Theta \in (0,1).
	\end{align*}	 
	\noi  \textbf{Subcase 2:} when $(x,y) \in O_2$. \\
	Again in consequence of  \eqref{nh21}, we have 
	\begin{align*}
b(u_1+t u_\e)_{|O_2} & \geq  b(u_1)_{|O_2} +b(t u_\e)_{|O_2} +2.2^*_{\mu} t^{2.2^*_{\mu}-1} \int\int_{O_2}\frac{(u_\e(x))^{2^*_{\mu}}(u_\e(y))^{2^*_{\mu}-1}u_1(y)}{|x-y|^{\mu}}dxdy\\& \quad + 2.2^*_{\mu} t \iint_{O_2} \frac{(u_1(x))^{2^*_{\mu}}(u_1(y))^{2^*_{\mu}-1}u_\e(y)}{|x-y|^{\mu}}~dxdy - A^2_\e,
	\end{align*}
	where $A^2_\e$ is sum of eight non-negative integrals and  each integral  has an  upper bound of the form  $ C \int\int_{O_2}\frac{u_1(x)(tu_\e(x))^{2^*_{\mu}-1} (u_\e(y))^{2^*_{\mu}}}{|x-y|^{\mu}}~dxdy$ or $ C \int\int_{O_2}\frac{(u_1(y))^{2^*_{\mu}-1}(tu_\e(y))(u_1(x))^{2^*_{\mu}}}{|x-y|^{\mu}}~dxdy$.
	By the similar estimates as in Subcase 1 and taking in account the definition of $O_2$ and the fact that  $u_1\in L^{\infty}(\Om)$,  we have  
	\begin{align*}
	\int\int_{O_2}\frac{u_1(x)(tu_\e(x))^{2^*_{\mu}-1} (u_\e(y))^{2^*_{\mu}}}{|x-y|^{\mu}}~dxdy \leq O(\e^{(\frac{2N-\mu}{4}) \Theta}) \text{ for all } \Theta \in (0,1).
	\end{align*}
	Write $(u_1(y))^{2^*_{\mu}-1}= (u_1(y))^{r}. (u_1(y))^{s}$ with $2^*_{\mu}-1= r+s$ and $0<1+s<\frac{2^*_{\mu}}{2}$ then 
	in consequence of  the definition of $O_2$ and  the fact that $u_1\in L^{\infty}(\Om)$,  we have  the following estimates 
	\begin{equation*}
	\begin{aligned}
	\int\int_{O_2} \frac{(u_1(y))^{2^*_{\mu}-1}(tu_\e(y))(u_1(x))^{2^*_{\mu}}}{|x-y|^{\mu}}~dxdy
&	\leq \int\int_{O_2}\frac{(u_1(y))^{r}(tu_\e(y))^{1+s}(u_1(x))^{2^*_{\mu}}}{|x-y|^{\mu}}~dxdy\\
	&\leq  C \int_{\Om}\int_{\Om}\frac{(tu_\e(y))^{1+s}(u_1(x))^{2^*_{\mu}}}{|x-y|^{\mu}}~dxdy\\
	&\leq  C \int_{\Om}\int_{\Om}\frac{  \e^{\frac{(1+s)(N-2)}{2}}   }{|x-y|^{\mu} |y|^{(1+s)(N-2)}}~dxdy \\
	&\leq  C   \e^{\frac{(1+s)(N-2)}{2}} \left( \ds \int_{\Om}
	\frac{dy}{|y|^{\frac{(1+s)(2N)(N-2)}{2N-\mu}}}
	\right)^{\frac{2N-\mu}{2N}}\\
	& \leq C \e^{\frac{(1+s)(N-2)}{2}} \left( \ds \int_{\Om}
	\frac{dy}{|y|^{\frac{(1+s)(2N)(N-2)}{2N-\mu}}}
	\right)^{\frac{2N-\mu}{2N}}.
	\end{aligned}
	\end{equation*}
	By the choice of  $s$, we have  $ \ds \int_{\Om}
	\frac{dx}{|x|^{\frac{(1+s)(2N)(N-2)}{2N-\mu}}} < \infty$. Hence we obtain 
	\begin{align*}
	\int\int_{O_2}\frac{(u_1(y))^{2^*_{\mu}-1}(tu_\e(y))(u_1(x))^{2^*_{\mu}}}{|x-y|^{\mu}}~dxdy \leq O(\e^{(\frac{2N-\mu}{4}) \Theta}) \text{ for all } \Theta \in (0,1).
	\end{align*}	 
	\noi  \textbf{Subcase 3:} when $(x,y) \in O_3$. \\
	Again from  \eqref{nh21}, we have  the following inequality
	\begin{align*}
	b(u_1+t u_\e)_{|O_3} & \geq  b(u_1)_{|O_3} +b(t u_\e)_{|O_3} +2.2^*_{\mu} t^{2.2^*_{\mu}-1} \int\int_{O_3}\frac{(u_\e(x))^{2^*_{\mu}}(u_\e(y))^{2^*_{\mu}-1}u_1(y)}{|x-y|^{\mu}}dxdy\\& \quad + 2.2^*_{\mu} t \int\int_{O_3} \frac{(u_1(x))^{2^*_{\mu}}(u_1(y))^{2^*_{\mu}-1}u_\e(y)}{|x-y|^{\mu}}~dxdy - A^3_\e,
	\end{align*}
	where $A^3_\e$ is sum of eight non-negative integrals and  each integral  has an  upper bound of the form $ C \int\int_{O_3}\frac{(u_1(x))^{2^*_{\mu}-1}(tu_\e(x)) (u_1(y))^{2^*_{\mu}}}{|x-y|^{\mu}}~dxdy$ or  $ C \int\int_{O_3}\frac{u_1(y)(tu_\e(y))^{2^*_{\mu}-1} (u_\e(x))^{2^*_{\mu}}}{|x-y|^{\mu}}~dxdy$.
	By the similar estimates as in Subcase 2, definition of $O_3$ and  regularity of $u_1$,  we have  
	\begin{align*}
	\int\int_{O_3}\frac{(u_1(x))^{2^*_{\mu}-1}(tu_\e(x)) (u_1(y))^{2^*_{\mu}}}{|x-y|^{\mu}}~dxdy \leq O(\e^{(\frac{2N-\mu}{4}) \Theta}) \text{ for all } \Theta \in (0,1).
	\end{align*}
	Also adopting the  estimates as in Subcase 1,  we have  
	\begin{align*}
	\int\int_{O_3}\frac{u_1(y)(tu_\e(y))^{2^*_{\mu}-1} (u_\e(x))^{2^*_{\mu}}}{|x-y|^{\mu}}~dxdy  \leq O(\e^{(\frac{2N-\mu}{4}) \Theta}) \text{ for all } \Theta\in (0,1).
	\end{align*}
	\noi  \textbf{Subcase 4:} when $(x,y) \in O_4$. \\
	As a result of \eqref{nh21}, we have 
	\begin{align*}
b(u_1+t u_\e)_{|O_4} & \geq  b(u_1)_{|O_4} +b(t u_\e)_{|O_4} +2.2^*_{\mu} t^{2.2^*_{\mu}-1} \int\int_{O_4}\frac{(u_\e(x))^{2^*_{\mu}}(u_\e(y))^{2^*_{\mu}-1}u_1(y)}{|x-y|^{\mu}}dxdy\\& \quad + 2.2^*_{\mu} t \int\int_{O_4} \frac{(u_1(x))^{2^*_{\mu}}(u_1(y))^{2^*_{\mu}-1}u_\e(y)}{|x-y|^{\mu}}~dxdy - A^4_\e,
	\end{align*}
	where $A^4_\e$ is sum of eight non-negative integrals and  each integral  has an  upper bound of the form $ C \int\int_{O_4}\frac{(u_1(x))^{2^*_{\mu}-1}(tu_\e(x)) (tu_\e(y))^{2^*_{\mu}}}{|x-y|^{\mu}}~dxdy$ or  $ C \int\int_{O_4}\frac{u_1(y)(tu_\e(y))^{2^*_{\mu}-1} (u_\e(x))^{2^*_{\mu}}}{|x-y|^{\mu}}~dxdy$.
	By the similar estimates as in Subcase 2, we have  
	\begin{align*}
	A^4_\e \leq O(\e^{(\frac{2N-\mu}{4}) \Theta}) \text{ for all } \Theta \in (0,1).
	\end{align*}
	From all subcases we obtain $ A^i_\e \leq O(\e^{(\frac{2N-\mu}{4}) \Theta}) \text{ for all } \Theta \in (0,1) \; \text{ and } i=1,2,3,4.$
	Combining all the subcases we get \eqref{nh22}. It completes the proof of claim.  Now 	 
	Combining  case 1 and case 2,  we get 
	\begin{align*}
	J_\la(u_1+t u_\e)  \leq & \frac{1}{2}a\|u_1+t u_\e\|^2 + \frac{\e^p}{2\theta}\|u_1+t u_\e\|^{2\theta} -\int_{\Om}f(x)|u_1+t u_\e|^q~dx\\&   -\widehat{C}t^{2.2^*_{\mu}-1} \int_{\Om}\int_{ \Om}\frac{(u_\e(x))^{2^*_{\mu}}(u_\e(y))^{2^*_{\mu}-1}u_1(y)}{|x-y|^{\mu}}~dxdy - \frac{1}{2.2^*_{\mu}} \|u_1\|_{NL}^{2.2^*_{\mu}} \\& 
	-\frac{1}{2.2^*_{\mu}} \|t u_\e\|_{NL}^{2.2^*_{\mu}} 	- t\int_{\Om}\int_{ \Om}\frac{(u_1(x))^{2^*_{\mu}}(u_1(y))^{2^*_{\mu}-1}u_\e(y)}{|x-y|^{\mu}}~dxdy+ O(\e^{(\frac{2N-\mu}{4}) \Theta}),
	\end{align*}
	for all  $\Theta <1.$ Taking $\Theta = \frac{2}{2^*_{\mu}}$, we have 
	\begin{align*}
	J_\la(u_1+t u_\e)  \leq & \frac{1}{2}a\|u_1+t u_\e\|^2 + \frac{\e^p}{2\theta}\|u_1+t u_\e\|^{2\theta} -\int_{\Om}f(x)|u_1+t u_\e|^q~dx\\&   -\widehat{C}t^{2.2^*_{\mu}-1} \int_{\Om}\int_{ \Om}\frac{(u_\e(x))^{2^*_{\mu}}(u_\e(y))^{2^*_{\mu}-1}u_1(y)}{|x-y|^{\mu}}~dxdy - \frac{1}{2.2^*_{\mu}} \|u_1\|_{NL}^{2.2^*_{\mu}} \\& 
	-\frac{1}{2.2^*_{\mu}} \|t u_\e\|_{NL}^{2.2^*_{\mu}} 	- t\int_{\Om}\int_{ \Om}\frac{(u_1(x))^{2^*_{\mu}}(u_1(y))^{2^*_{\mu}-1}u_\e(y)}{|x-y|^{\mu}}~dxdy+ o(\e^{\frac{N-2}{2} }) .
	\end{align*}
	Observe that $\bigg|\ds \int_{ \Om}\na u_1\cdot \na tu_\e~dx\bigg| \leq \|u_1\|\|tu_\e\|$. Therefore for some $\al \in [0,2\pi]$, we have 
	\begin{align*}
\int_{ \Om}\na u_1\cdot \na tu_\e~dx= \|u_1\|\|tu_\e\| \cos \al.
	\end{align*}
	It implies 
	\begin{align*}
	\|u_1+t u_\e\|^{2\theta}& = \left(
	\|u_1\|^2+\|tu_\e\|^2+2\int_{ \Om}\na u_1\cdot \na tu_\e~dx\right)^\theta\\& = \left(
	\|u_1\|^2+\|tu_\e\|^2+2\|u_1\|\|tu_\e\| \cos \al \right)^\theta.
	\end{align*}
	Now we use the following one-dimensional inequality: for all $y\geq 0,\; \al \in [0,2\pi]$, there exists a uniform $R>0$ such that 
	\begin{equation}\label{kh8}
	\begin{aligned}
	(1+y^2+2y\cos\al)^\theta \leq  \left\{
	\begin{array}{ll}
	1+y^{2\theta}+2\theta y \cos\al +Ry^{2}, &  \text{ if }1\leq \theta < \frac{3}{2},\\
	1+y^{2\theta}+2\theta y \cos\al +R(y^{2\theta-1}+y^2) & \text{ if } \theta\geq \frac{3}{2}.\\
	\end{array} 
	\right.
	\end{aligned}	
	\end{equation}
	Using \eqref{kh8} with $y= \frac{\|tu_\e\|}{\|u_1\|}$, we have the following uniform estimate
	\begin{equation}
	\begin{aligned}\label{kh11}
	\|u_1+t u_\e\|^{2\theta}&= \left(
	\|u_1\|^2+\|tu_\e\|^2+2\|u_1\|\|tu_\e\| \cos \al \right)^\theta\\
	& \leq \|u_1\|^{2\theta} + \|t u_\e\|^{2\theta} + 2\theta  \|u_1\|^{2\theta-1}\|t u_\e\| \cos\al +  R\|u_1\|\left(\|u_\e\|^2 + \|u_\e\|^{2\theta-1}\right)\\
	&=  \|u_1\|^{2\theta} + \|t u_\e\|^{2\theta} + 2\theta  \|u_1\|^{2\theta-2} \int_{ \Om}\na u_1\cdot \na tu_\e~dx + R C_0\left(\|u_\e\|^2 + \|u_\e\|^{2\theta-1}\right)\\
	\end{aligned}
	\end{equation}
where $C_0= \|u_1\|$.
	Employing \eqref{kh11}, we obtain the subsequent estimates 
	\begin{align*}
	J_\la(u_1+t u_\e)  \leq & \frac{a}{2}\|u_1\|^2+  \frac{a}{2}\|t u_\e\|^2 +a \int_{ \Om}\na u_1\cdot \na tu_\e~dx+ \frac{\e^p}{2\theta}\|u_1\|^{2\theta}+ \frac{\e^p}{2\theta}\|t u_\e\|^{2\theta}\\
	&  + \e^p\|u_1\|^{2\theta-2} \int_{ \Om}\na u_1\cdot \na tu_\e~dx+ \e^p   R C_0\left(\|u_\e\|^2 + \|u_\e\|^{2\theta-1}\right)\\
	& -\int_{\Om}f(x)|u_1+t u_\e|^q~dx   - \frac{1}{2.2^*_{\mu}} \|u_1\|_{NL}^{2.2^*_{\mu}}  
	-\frac{1}{2.2^*_{\mu}} \|t u_\e\|_{NL}^{2.2^*_{\mu}} \\
	&  -\widehat{C}t^{2.2^*_{\mu}-1} \int_{\Om}\int_{ \Om}\frac{(u_\e(x))^{2^*_{\mu}}(u_\e(y))^{2^*_{\mu}-1}u_1(y)}{|x-y|^{\mu}}~dxdy \\
	& 	- t\int_{\Om}\int_{ \Om}\frac{(u_1(x))^{2^*_{\mu}}(u_1(y))^{2^*_{\mu}-1}u_\e(y)}{|x-y|^{\mu}}~dxdy+ o(\e^{\frac{N-2}{2} }).
	\end{align*}	
	Now making use of the facts that $u_1$ solves $(P_\la)$ and $J_\la(u_1)<0$, we have 
	\begin{align*}
	J_\la(u_1+t u_\e)  \leq &  J_\la(u_1) +\frac{a}{2}\|t u_\e\|^2 -\frac{1}{2.2^*_{\mu}} \|t u_\e\|_{NL}^{2.2^*_{\mu}}     \\&-\widehat{C}t^{2.2^*_{\mu}-1} \int_{\Om}\int_{ \Om}\frac{(u_\e(x))^{2^*_{\mu}}(u_\e(y))^{2^*_{\mu}-1}u_1(y)}{|x-y|^{\mu}}~dxdy  
	  -\int_{\Om}f(x)|u_1+t u_\e|^q~dx\\& +\la \int_{\Om}|u_1|^q~dx + q\la t\int_{\Om}u_1^{q-1}u_\e~dx  + \frac{\e^p}{2\theta}\|t u_\e\|^{2\theta} \\& + \e^p  C_0 R\|t u_\e\|^{2\theta-1}+\e^pC_0 R \|t u_\e\|^2 + o(\e^{\frac{N-2}{2} }) \\
	 \leq& \frac{a}{2}\|t u_\e\|^2 -\frac{1}{2.2^*_{\mu}} \|t u_\e\|_{NL}^{2.2^*_{\mu}}     -\widehat{C}t^{2.2^*_{\mu}-1} \int_{\Om}\int_{ \Om}\frac{(u_\e(x))^{2^*_{\mu}}(u_\e(y))^{2^*_{\mu}-1}u_1(y)}{|x-y|^{\mu}}~dxdy  \\
	&  -\int_{\Om}f(x) \left( \int_{0}^{t u_\e}|u_1+s|^q - |u_1|^q~ ds \right)~dx  \\
	&  + \frac{\e^p}{2\theta}\|t u_\e\|^{2\theta}  + \e^p  C_0 R\|t u_\e\|^{2\theta-1}+\e^pC_0 R \|t u_\e\|^2 + o(\e^{\frac{N-2}{2} }). 
	\end{align*}	
From $f>0$ in $\Sigma$ and $tu_\e =0$ in $\Sigma^c$ and using claim 1, we see that 
	\begin{align*}
	J_\la(u_1+t u_\e) 
	&  \leq  \frac{at^2}{2}\| u_\e\|^2-\frac{t^{2.2^*_{\mu}}}{2.2^*_{\mu}} \| u_\e\|_{NL}^{2.2^*_{\mu}}+ \frac{\e^pt^{2\theta}}{2\theta}\| u_\e\|^{2\theta} \\
	& \quad - t^{2.2^*_{\mu}-1} \widehat{C}r_1\e^{\frac{N-2}{2}}  + \e^p  C_0 R\|t u_\e\|^{2\theta-1}+\e^pC_0 R \|t u_\e\|^2 + o(\e^{\frac{N-2}{2} }) .
	\end{align*}
We define 
\begin{align*}
\mathcal{H}(t) :=&   \frac{at^2}{2}\| u_\e\|^2 + \frac{\e^pt^{2\theta}}{2\theta}\| u_\e\|^{2\theta} -\frac{t^{2.2^*_{\mu}}}{2.2^*_{\mu}} \| u_\e\|_{NL}^{2.2^*_{\mu}}\\
& + \e^p  C_0 R\|t u_\e\|^{2\theta-1}+\e^pC_0 R \|t u_\e\|^2 + - t^{2.2^*_{\mu}-1}\widehat{C}r_1\e^{\frac{N-2}{2}}. 
\end{align*}
	Then $\mathcal{H}(t)\ra \infty$ as $t\ra \infty$ and $\ds \lim_{t\ra 0^+}\mathcal{H}(t)>0$. Hence there exists a $t_\e>0$ such that $\ds \sup_{t>0}\mathcal{H}(t)= \mathcal{H}(t_\e)$ and 
\begin{equation}
	\begin{aligned}\label{kh10}
	 0= \mathcal{H}^\prime(t_\e)=
	 &  t_\e a\| u_\e\|^2 + \e^pt_\e  ^{2\theta-1}\| u_\e\|^{2\theta} -t_\e ^{2.2^*_{\mu}-1} \| u_\e\|_{NL}^{2.2^*_{\mu}}\\
	 & +\e^p  C_0 R t_\e ^{2\theta-2} \| u_\e\|^{2\theta-1}+\e^pC_0 R t_\e  \| u_\e\|^2 - t_\e ^{2.2^*_{\mu}-2}\widehat{C}r_1\e^{\frac{N-2}{2}} . 
	\end{aligned}
	\end{equation}
	From \eqref{kh10}, we have 
	\begin{align*}
	t_\e ^{2.2^*_{\mu}-2} \| u_\e\|_{NL}^{2.2^*_{\mu}} \leq a \| u_\e\|^2 + \e^pt_\e ^{2\theta-2}\| u_\e\|^{2\theta}+\e^p  C_0 R t_\e ^{2\theta-2} \| u_\e\|^{2\theta-1}+\e^pC_0 R t_\e  \| u_\e\|^2,
	\end{align*}
	Now using the fact that $2^*_{\mu}>\theta$, there exists a $T_0>0$ such that $t_\e<T_0$. Again by \eqref{kh10}, we have
	\begin{align*}
	 a\| u_\e\|^2 \leq t_\e ^{2.2^*_{\mu}-2} \| u_\e\|_{NL}^{2.2^*_{\mu}}+ t_\e ^{2.2^*_{\mu}-3}\widehat{C}r_1\e^{\frac{N-2}{2}}.
	\end{align*}
	It implies that there exists a $T_{00}>0$ such that $T_{00}<t_\e$. Now let  
	\begin{align}\label{kh13}
	\mathcal{G}(t)= \frac{at^2}{2}\| u_\e\|^2 -\frac{t^{2.2^*_{\mu}}}{2.2^*_{\mu}} \| u_\e\|_{NL}^{2.2^*_{\mu}}.
	\end{align}
	Then $\mathcal{G}(t)$  attains its maximum at $t_0= \left(\frac{a\| u_\e\|^2}{\| u_\e\|_{NL}^{2.2^*_{\mu}}}\right)^{\frac{1}{2.2^*_{\mu}-2}}$.
	
	 Therefore using Proposition \ref{khprop3},
	\begin{align*}
	\sup_{t>0}\mathcal{H}(t)& \leq \mathcal{G}(t_\e)+ \frac{\e^pT_0^{2\theta}}{2\theta}\| u_\e\|^{2\theta} + \e^p  C_0 T_0^{2\theta-1} R\| u_\e\|^{2\theta-1}+\e^pC_0 R T_0^{2} \| u_\e\|^2  - T_{00}^{2.2^*_{\mu}-1}\widehat{C}r_1\e^{\frac{N-2}{2}}\\
	&\leq \mathcal{G}(t_0)+ \frac{\e^pT_0^{2\theta}}{2\theta}\| u_\e\|^{2\theta} + \e^p  C_0 T_0^{2\theta-1} R\| u_\e\|^{2\theta-1}+\e^pC_0 R T_0^{2} \| u_\e\|^2   - T_{00}^{2.2^*_{\mu}-1}\widehat{C}r_1\e^{\frac{N-2}{2}}\\
	&=  \frac{N-\mu+2}{2(2N-\mu)}\left(aS_{H,L}\right)^{\frac{2N-\mu}{N-\mu+2}}+O(\e^{N-2})+ \frac{\e^pT_0^{2\theta}}{2\theta}\| u_\e\|^{2\theta} + \e^p  C_0 T_0^{2\theta-1} R\| u_\e\|^{2\theta-1}\\
	& \quad +\e^pC_0 R T_0^{2} \| u_\e\|^2   - T_{00}^{2.2^*_{\mu}-1}\widehat{C}r_1\e^{\frac{N-2}{2}}. 
	\end{align*}
	Thus 
	\begin{align*}
	J_\la(u_1+t u_\e) 
	&  \leq  \frac{N-\mu+2}{2(2N-\mu)}\left(aS_{H,L}\right)^{\frac{2N-\mu}{N-\mu+2}}+C_1\e^{N-2}+ \frac{\e^pT_0^{2\theta}}{2\theta}\| u_\e\|^{2\theta} + \e^p  C_0 T_0^{2\theta-1} R\| u_\e\|^{2\theta-1}\\
	& \quad +\e^pC_0 R T_0^{2} \| u_\e\|^2  - T_{00}^{2.2^*_{\mu}-1}\widehat{C}r_1\e^{\frac{N-2}{2}}  + o(\e^{\frac{N-2}{2} }), 
	\end{align*}
	for some $C_1>0$.	Since $p>N-2$, it implies there exists a constant $C_2>0$ such that
	\begin{align*}
	C_1\e^{N-2}+ \frac{\e^pT_0^{2\theta}}{2\theta}\| u_\e\|^{2\theta} + \e^p  C_0 T_0^{2\theta-1} R\| u_\e\|^{2\theta-1}+\e^pC_0 R T_0^{2} \| u_\e\|^2 + o(\e^{\frac{N-2}{2} }) \leq C_2\e^{\beta}, 
	\end{align*}
	where $\beta > \frac{N-2}{2}$. Hence 
	 \begin{align*}
	 J_\la(u_1+t u_\e) 
	 &  \leq  \frac{N-\mu+2}{2(2N-\mu)}\left(aS_{H,L}\right)^{\frac{2N-\mu}{N-\mu+2}}+C_2\e^{\beta} - T_{00}^{2.2^*_{\mu}-1}\widehat{C}r_1\e^{\frac{N-2}{2}} .
	 \end{align*}
	 Let $\e= (\la^{\frac{2}{2-q}})^{\frac{1}{\beta}}$ then 
	 \begin{align*}
	 J_\la(u_1+t u_\e) 
	 &  \leq  \frac{N-\mu+2}{2(2N-\mu)}\left(aS_{H,L}\right)^{\frac{2N-\mu}{N-\mu+2}}+C_2\la^{\frac{2}{2-q}} - T_{00}^{2.2^*_{\mu}-1}\widehat{C}r_1(\la^{\frac{2}{2-q}})^{\frac{N-2}{2\beta}} .
	 \end{align*}
	 Now let $\La^*_1=  \left(\frac{T_{00}^{2.2^*_{\mu}-1}\widehat{C}r_1}{\ds C_2+\widehat{D}}\right)^{\frac{(2-q)\beta}{(2\beta-(N-2))}}$, where $\widehat{D}$ is defined in Proposition \ref{khprop2}. Then for all $0<\la<\La^*_1$,  we have 
	 \begin{align*}
	 C_2\la^{\frac{2}{2-q}} - T_{00}^{2.2^*_{\mu}-1}\widehat{C}r_1(\la^{\frac{2}{2-q}})^{\frac{N-2}{2\beta}} = \la^{\frac{2}{2-q}}\left(C_2 - T_{00}^{2.2^*_{\mu}-1}\widehat{C}r_1\la^{\frac{2}{2-q}\left(\frac{N-2}{2\beta}-1\right)}\right) < - \widehat{D} \la^{\frac{2}{2-q}}.
	 \end{align*}
	 Define $\Upsilon^*= \min\{\La^*,\La^*_1\}$ and $\e^*= (\Upsilon^*)^{\frac{2}{\beta(2-q)}}>0 $ such that 	 
	 for every  $\la \in (0, \Upsilon^{*})$ and $\e \in (0, \e^{*})$, we have 
	 \begin{align*}
	 \sup_{t\geq 0}J_\la(u_1+t u_\e)< c_{\infty}.
	 \end{align*}
	 \QED
\end{proof}	

 \begin{Lemma}\label{khlem8}
If $\mu<4$, $\la \in (0,\Upsilon^*)$ and $\e \in (0,\e^*)$ then the following holds:
	\begin{itemize}
		\item [(i)] $H_0^1(\Om)\setminus N_\la^-= U_1\cup U_2$,
		where 
		\begin{align*}
		& U_1:= \left \lbrace u \in H_0^1(\Om)\setminus \{0\}\; \middle|\;   \|u\|<t^-\left(\frac{u}{\|u\|}\right)  \right \rbrace \cup \{0\},\\
		& U_2:= \left \lbrace u \in H_0^1(\Om)\setminus \{0\}\; \middle|\;  \|u\|>t^-\left(\frac{u}{\|u\|}\right)  \right \rbrace. 
		\end{align*}
		\item [(ii)] $N_\la^+ \subset U_1$.
		\item [(iii)] There exists $t_0>1$  such that $u_1+t_0u_
		\e \in U_2$.
			\item [(iv)] There exists $s_0\in (0,1)$  such that $u_1+ s_0  u_\e\in N_\la^-$.
		\item [(v)] $\theta_\la^- < c_\infty$.
	\end{itemize}
\end{Lemma}
\begin{proof}
	\begin{itemize}
		\item [(i)] It holds by Lemma \ref{khlem5} (iv).
		\item[(ii)] Let $u \in N_\la^+$ then $t^+(u)=1$. So, $1<t^+(u)<t_{max}<t^-(u)= \frac{1}{\|u\|}t^-\left(\frac{u}{\|u\|}\right)$ that is,  $N_\la^+ \subset U_1$.
		\item [(iii)] First, we will show that there exists a constant $c>0$ such that $0<t^-\left(\frac{u_1+tu_\e}{\|u_1+tu_\e\|}\right)< c $ for all $t>0$. On the contrary let there exist a sequence $\{t_n\}$ such that $t_n\ra \infty$ and $t^-\left(\frac{u_1+t_nu_\e}{\|u_1+t_nu_\e\|}\right)\ra \infty$ as $n\ra \infty$. Let $u_n:= \frac{u_1+t_nu_\e}{\|u_1+t_nu_\e\|}$, then by the fibering analysis,  $t^-(u_n)u_n \in N_\la^-$ and  by dominated convergence theorem,
		\begin{align*}
		\|u_n\|_{NL}^{2.2^*_{\mu}}= \frac{\|u_1+t_nu_\e\|_{NL}^{2.2^*_{\mu}}}{\|u_1+t_nu_\e\|^{2.2^*_{\mu}}}= \frac{\|\frac{u_1}{t_n} +u_\e\|_{NL}^{2.2^*_{\mu}}}{\|\frac{u_1}{t_n} +u_\e\|^{2.2^*_{\mu}}}\ra \frac{\|u_\e\|_{NL}^{2.2^*_{\mu}}}{\|u_\e\|^{2.2^*_{\mu}}} \text{ as } n\ra \infty.
		\end{align*}
		Hence, $J_\la(t^-(u_n)u_n)\ra -\infty$ as $n\ra \infty$, contradicts the fact that $J_\la$ is bounded below on $N_\la$. Thus,  there exists $c>0$ such that $0<t^-\left(\frac{u_1+tu_\e}{\|u_1+tu_\e\|}\right)< c $ for all $t>0$. Let $\ds t_0= \frac{|c^2-\|u_1\|^2|^{\frac{1}{2}}}{\|u_\e\|}+1$ then
		\begin{align*}
		\|u_1+t_0u_\e\|^2 & =\|u_1\|^2+t_0^2\|u_\e\|^2+2t_0\ld u_1, \; u_\e \rd \\
		& \geq \|u_1\|^2+ |c^2-\|u_1\|^2| \geq c^2\geq \left(t^-\left(\frac{u_1+tu_\e}{\|u_1+tu_\e\|}\right)\right)^2.
		\end{align*}
		It implies that $u_1+t_0u_\e\in U_2$.
		\item [(iv)] For every  $\la \in (0,\Upsilon^*)$ and $\e \in (0,\e^*)$, define a path $\xi_{\e}(s)= u_1+ s t_0u_\e$ for $s\in [0,1]$. Then 
	$	\xi_{\e}(0)= u_1 \quad \text{and} \quad \xi_{\e}(1)= u_1+ t_0u_\e \in U_2.$
		Since $\frac{1}{\|u\|}t^-\left(\frac{u}{\|u\|}\right)$ is a continuous function and $\xi_{\e}([0,1])$ is connected. So, there exists $s_0 \in [0,1]$ such that $\xi_{\e}(s_0)=u_1+ s_0t_0u_\e \in N_\la^- $.
		\item [(v)] Using part (d) and Lemma  \ref{khlem9}.\QED
	\end{itemize}
\end{proof}

\begin{Lemma}\label{khlem10}
		Let $\mu\geq \min\{4,N\}$ and $\frac{N}{N-2}\leq q <2$ then there exist  $\Upsilon^{**}>0$ and $\e^{**}>0$ such that for every $\la \in (0, \Upsilon^{**})$ and $\e \in (0, \e^{**})$, we have 
	\begin{align*}
	\sup_{r\geq 0}J_\la(r u_\e)< c_{\infty}.
	\end{align*}
	Moreover, we have  $\theta_\la^-<c_\infty$.
\end{Lemma}
\begin{proof}  Let $0< \la< \La^*$ then $c_\infty >0$ and  \begin{align*}
J_\la(r u_\e)\leq  
\ds \frac{ar^2}{2}\| u_\e\|^2+ \frac{\e^pr^{2\theta}}{2\theta}\| u_\e\|^{2\theta} \leq C (r^2+ r^{2\theta}).
	\end{align*} 
	Therefore there exists a $r_0 \in (0,1)$ such that 
	$
	\sup_{0\leq r\leq r_0}J_\la(r u_\e) < c_\infty, 
	$
	for all $0<\la< \La^*$. This implies we only have to show that $	\ds \sup_{ r\geq r_0}J_\la(r u_\e) < c_\infty$. Now consider 
	\begin{align*}
\sup_{ r\geq r_0}	J_\la(r u_\e) & = \sup_{ r\geq r_0}\left( \frac{ar^2}{2}\| u_\e\|^2+ \frac{\e^pr^{2\theta}}{2\theta}\| u_\e\|^{2\theta} - \frac{\la r^q}{q} \int_{ \Om}f(x) |u_\e|^q~dx - \frac{r^{2.2^*_{\mu}}}{2.2^*_{\mu}} \|u_\e\|_{NL}^{2.2^*_{\mu}}\right)\nonumber\\
& = \sup_{ r\geq r_0}\left( \frac{ar^2}{2}\| u_\e\|^2+ \frac{\e^pr^{2\theta}}{2\theta}\| u_\e\|^{2\theta} - \frac{\la r^q}{q} \int_{B_{2\de}(0)}f(x) |u_\e|^q~dx - \frac{r^{2.2^*_{\mu}}}{2.2^*_{\mu}} \|u_\e\|_{NL}^{2.2^*_{\mu}}\right)\nonumber\\
&  \leq  \sup_{ r\geq r_0}\left( \frac{ar^2}{2}\| u_\e\|^2+ \frac{\e^pr^{2\theta}}{2\theta}\| u_\e\|^{2\theta} - \frac{m_f\la r^q}{q} \int_{B_{2\de}(0)} |u_\e|^q~dx - \frac{r^{2.2^*_{\mu}}}{2.2^*_{\mu}} \|u_\e\|_{NL}^{2.2^*_{\mu}}\right)\nonumber\\
& \leq  \sup_{ r\geq r_0}\left( \frac{ar^2}{2}\| u_\e\|^2+ \frac{\e^pr^{2\theta}}{2\theta}\| u_\e\|^{2\theta} - \frac{m_f\la r^q}{q} \int_{B_{\de}(0)} |U_\e|^q~dx - \frac{r^{2.2^*_{\mu}}}{2.2^*_{\mu}} \|u_\e\|_{NL}^{2.2^*_{\mu}}\right)\nonumber\\
& \leq  \sup_{ r\geq 0} \mathcal{V}(r) - \frac{m_f\la r_0^q}{q} \int_{B_{\de}(0)} |U_\e|^q~dx,
	\end{align*}
	where $\mathcal{V}(r)= \ds \frac{ar^2}{2}\| u_\e\|^2+ \frac{\e^pr^{2\theta}}{2\theta}\| u_\e\|^{2\theta} - \frac{r^{2.2^*_{\mu}}}{2.2^*_{\mu}} \|u_\e\|_{NL}^{2.2^*_{\mu}}$. Then $\mathcal{V}(0)=0$, $\ds \lim_{r\ra 0^+} \mathcal{V}(r)>0$ and 
 $\mathcal{V}(r)\ra -\infty$ as $r\ra \infty$. Therefore, there exists a $r_\e>0$ such that $\ds \sup_{r>0}\mathcal{V}(r)= \mathcal{V}(r_\e)$ and 
\begin{equation}\label{kh12}
\begin{aligned}
0= \mathcal{V}^\prime(r_\e)=
&  r_\e a\| u_\e\|^2 + \e^pr_\e^{2\theta-1}\| u_\e\|^{2\theta}  -r_\e^{2.2^*_{\mu}-1}\| u_\e\|_{NL}^{2.2^*_{\mu}} . 
\end{aligned}
\end{equation}
From \eqref{kh12}, we have 
\begin{align*}
r_\e^{2.2^*_{\mu}-2}= \frac{1}{\| u_\e\|_{NL}^{2.2^*_{\mu}}}\left(a \| u_\e\|^2 + \e^pr_\e^{2\theta-2}\| u_\e\|^{2\theta} \right)<C(1+r_\e^{2\theta-2}),
\end{align*}
 for some $C>0$. Using the fact that $2^*_{\mu}>\theta$, there exists a $r_1>0$ such that $r_\e<r_1$. Combining all these, we get 
 \begin{align*}
 \sup_{ r\geq r_0}	J_\la(r u_\e) &  \leq  \sup_{ r\geq 0} \mathcal{V}(r) - \frac{m_f\la r_0^q}{q} \int_{B_{\de}(0)} |U_\e|^q~dx\nonumber\\
 & =\frac{ar_\e^2}{2}\| u_\e\|^2+ \frac{\e^pr_\e^{2\theta}}{2\theta}\| u_\e\|^{2\theta} - \frac{r_\e^{2.2^*_{\mu}}}{2.2^*_{\mu}} \|u_\e\|_{NL}^{2.2^*_{\mu}} - \frac{m_f\la r_0^q}{q} \int_{B_{\de}(0)} |U_\e|^q~dx \nonumber \\
 &   \leq  \sup_{ r\geq 0} \left(\frac{ar^2}{2}\| u_\e\|^2 - \frac{r^{2.2^*_{\mu}}}{2.2^*_{\mu}} \|u_\e\|_{NL}^{2.2^*_{\mu}} \right) + \frac{\e^pr_1^{2\theta}}{2\theta}\| u_\e\|^{2\theta} - \frac{m_f\la r_0^q}{q} \int_{B_{\de}(0)} |U_\e|^q~dx \nonumber\\
 &=\sup_{r\ge 0} \mathcal{G}(r) +\frac{\e^pr_1^{2\theta}}{2\theta}\| u_\e\|^{2\theta} - \frac{m_f\la r_0^q}{q} \int_{B_{\de}(0)} |U_\e|^q~dx ,
 \end{align*}
 where $\mathcal{G}(r)$ is defined as in \eqref{kh13}. Now since $\mathcal{G}(r)$ has maximum at $ r^*= \left(\frac{a\| u_\e\|^2}{\| u_\e\|_{NL}^{2.2^*_{\mu}}}\right)^{\frac{1}{2.2^*_{\mu}-2}}, $ 
 \begin{align}\label{kh14}
 \sup_{ r\geq r_0}	J_\la(r u_\e)& \le \frac{N-\mu+2}{2(2N-\mu)}\left(aS_{H,L}\right)^{\frac{2N-\mu}{N-\mu+2}}+O(\e^{N-2}) + \frac{\e^pr_1^{2\theta}}{2\theta}\| u_\e\|^{2\theta} - \frac{m_f\la r_0^q}{q} \int_{B_{\de}(0)} |U_\e|^q~dx \nonumber \\
 & \leq \frac{N-\mu+2}{2(2N-\mu)}\left(aS_{H,L}\right)^{\frac{2N-\mu}{N-\mu+2}}+C_1\e^{N-2} - \frac{m_f\la r_0^q}{q} \int_{B_{\de}(0)} |U_\e|^q~dx , 
\end{align}
 where the last inequality comes from the fact that $p>N-2$. Now we will find the estimates on $\int_{B_{\de}(0)} |U_\e|^q~dx$. For $0<\e<\frac{\de}{2}$, we have
 \begin{equation}
 \begin{aligned}\label{kh15}
 \int_{B_{\de}(0)} |U_\e|^q~dx & \geq C_2 w_{N-1} \e^{N-\frac{N-2}{2}q} \int_{1}^{\frac{\de}{\e}} r^{N-\frac{N-2}{2}(q-1)}~dr\\
 & \backsimeq C_3 \left\{
 \begin{array}{ll}
 \e^{N-\frac{N-2}{2}q}, &  \text{ if }  q> \frac{N}{N-2}, \\
\e^{N-\frac{N-2}{2}q} |\text{log}\; \e|, & \text{ if } q= \frac{N}{N-2}.\\
 \end{array} 
 \right.
  \end{aligned}
  \end{equation}
Using  \eqref{kh14} and \eqref{kh15} with $\e = (\la^{\frac{2}{2-q}})^{\frac{1}{N-2}} $, it follows that 
\begin{equation*}
\begin{aligned}
\sup_{ r\geq r_0}	J_\la(r u_\e) &   \leq \frac{N-\mu+2}{2(2N-\mu)}\left(aS_{H,L}\right)^{\frac{2N-\mu}{N-\mu+2}}+C_1\la^{\frac{2}{2-q}} \\
& - C_3 \la  \left\{
\begin{array}{ll}
\la^{\frac{2}{(2-q)(N-2)}\left(N-\frac{N-2}{2}q\right)}, &  \text{ if }  q> \frac{N}{N-2}, \\
\la^{\frac{2}{(2-q)(N-2)}\left(N-\frac{N-2}{2}q\right)}  |\text{\text{log}}\; (\la^{\frac{2}{2-q}})^{\frac{1}{N-2}}|, & \text{ if } q= \frac{N}{N-2}.\\
\end{array} 
\right.
\end{aligned}
\end{equation*}
\textbf{Case 1: } When $q> \frac{N}{N-2}$.\\
 Trivially, $q> \frac{N}{N-2}$ if and only if $1+  \frac{2}{(2-q)(N-2)}\left(N-\frac{N-2}{2}q\right) < \frac{2}{2-q}$. Thus we can choose $\gamma_1>0$ such that for every $0<\la<\gamma_1$ we have, 
 \begin{align*}
 C_1\la^{\frac{2}{2-q}}- C_3\la^{1+\frac{2}{(2-q)(N-2)}\left(N-\frac{N-2}{2}q\right)} < - \widehat{D} \la^{\frac{2}{2-q}},
 \end{align*}
 where $\widehat{D}$ is defined in Proposition \ref{khprop2}. \\
 \textbf{Case 2: } When $q= \frac{N}{N-2}$.\\
 As $\la \ra 0$ then $|\text{\text{log}}\; (\la^{\frac{2}{2-q}})^{\frac{1}{N-2}}|\ra \infty$ thus we can choose  $\gamma_2>0$ such that for every $0<\la<\gamma_2$, we have
 \begin{align*}
 C_1\la^{\frac{2}{2-q}}- C_3\la^{1+\frac{2}{(2-q)(N-2)}\left(N-\frac{N-2}{2}q\right)}|\text{\text{log}}\; (\la^{\frac{2}{2-q}})^{\frac{1}{N-2}}| < - \widehat{D} \la^{\frac{2}{2-q}},
 \end{align*}
By defining $\Upsilon^{**}= \min\{\La^*,\;\gamma_1,\;\gamma_2,\; (\de/2)^{N-2} \}>0
$ and  $\e^{**} = (\Upsilon^{**})^{\frac{2}{(2-q)(N-2)}}>0$,  we get 	\begin{align*}
\sup_{ r\geq 0}J_\la(r u_\e) < c_\infty, 
\end{align*} 
for all $\la \in (0, \Upsilon^{**})$ and $\e \in (0, \e^{**})$. Since there exists  $r_2>0$ such that $r_2 u_\e \in N_\la^-$. Thus 
\begin{align*}
\theta_\la^- \leq J_\la(r_2u_\e)\leq \sup_{ r\geq 0}J_\la(r u_\e) < c_\infty.
\end{align*}
\QED
 \end{proof}
 \textbf{Proof of Theorem \ref{khthm2} :}  From Remark \ref{khrem1} and Proposition \ref{khprop1}, there exists a minimizing sequence $\{u_n\} \subset N_{\la}^- $ such that
 \begin{align*}
 J_\la(u_{n}) = \theta_{\lambda}^-+o_n(1) \text{ and } J_\la^{\prime}(u_{n}) = o_n(1).
 \end{align*}
 If $\mu < \min \{4,N\}$ then from  Lemma \ref{khlem8},  for each  $\la \in (0,\Upsilon^*)$ and $\e \in (0,\e^*)$, we have  $\theta_\la^- < c_\infty$. If  $\mu \geq  \min \{4,N\}$ and $\frac{N}{N-2}\leq q<2$ then from  Lemma \ref{khlem10}, for every $\la \in (0, \Upsilon^{**})$ and $\e \in (0, \e^{**})$ we have  $\theta_\la^- < c_\infty$. This on using Proposition \ref{khprop2} gives that there exists a convergent subsequence  of $\{u_n\}$ (still denoted by $\{u_n\}$) and $u_2 \in H^1_0(\Om)$ such that $u_n\rightarrow u_2$ strongly in $H_0^1(\Om)$.  Since $N_\la^-$ is a closed set, $u_2 \in N_\la^-$ and $J_\la(u_2)= \theta_\la^-$ and also by Lemma \ref{khlem6},  $u_2$ is a solution $(P_\la)$. Since $J_\la(u_2)= J_\la(|u_2|)$, therefore $u_2$ is non-negative solution. By \cite[Lemma 4.4]{yangjmaa}  and strong maximum principle, we have $u_2$ is a positive solution of $(P_\la)$. Hence we get two positive solutions $u_1 \in N_\la^+$ and $u_2 \in N_\la^-$.\QED
 
  \section{ The case $q=2$}
 In this section,  we consider the problem $(P_\la)$ when $q=2$. Precisely we consider the problem:
 \begin{equation*}
 (P_\la)
 \left\{\begin{array}{rlll}
 -\left(  a+\e^p  \|\na u\|_{L^2}^{2\theta-2}\right)\De u
 &=\la f(x)u+ \left(\ds \int_{\Om}\frac{|u(y)|^{2^*_{\mu}}}{|x-y|^{\mu}}dy\right)|u|^{2^*_{\mu}-2}u,  \, \text{in}\,
 \Om,\\
 u&=0 \, \text{ on } \pa \Om. \end{array}\right.
 \end{equation*}
 The functional $J_\la$  is equal to 
 \begin{equation*}
 J_\la(u) = \frac{a}{2}\|u\|^2+\frac{\e^p}{2\theta} \|u\|^{2\theta}  - \frac{1}{2} \int_{\Om}f(x)|u|^2 ~dx- \frac{1}{2.2^*_{\mu}} \int_{\Om}\int_{\Om} \frac{|u(x)|^{2^{*}_{\mu}}|u(y)|^{2^{*}_{\mu}}}{|x-y|^{\mu}}~dxdy.
 \end{equation*}
 \begin{Lemma}\label{khlem11}
 	If $N\geq 3$
 	and $\la\in (0,\;aS\|f\|_{L^r}^{-1})$ 	then $J_\la$ satisfies the following conditions:
 	\begin{enumerate}
 		\item [(i)] There exists $\al,\rho>$ such that $J_\la(u)\geq \al$ for $\|u\|= \rho$.
 		\item [(ii)] There exists  $e \in H_0^1(\Om)$ with $\|e\|>\rho$ such that $J_\la(e)<0$.
 	\end{enumerate}
 \end{Lemma}
 
 \begin{proof}
 	(i) Using $\la\in (0,\;aS\|f\|_{L^r}^{-1})$, definition of $S$ and $S_{H,L}$, we have 
 	\begin{align*}
 	J_\la(u)\geq \frac{1}{2}(a-\la S^{-1}\|f\|_{L^r})\|u\|^2-\frac{S_{H,L}^{-1}}{2.2^*_{\mu}} \|u\|^{2.2^*_{\mu}}.
 	\end{align*}
 	Since $2^*_{\mu}>1$,  we can choose $\al,\rho>$ such that $J_\la(u)\geq \al$ for $\|u\|= \rho$. \\
 	(ii) Let $u \in H_0^1(\Om)$ then 
 	\begin{align*}
 	J_\la(tu) = &\frac{a t^2}{2}\|u\|^2+\frac{\e^p t^{2\theta}}{2\theta} \|u\|^{2\theta}  - \frac{t^2}{2} \int_{\Om}f(x)|u|^2 ~dx- \frac{t^{2.2^*_{\mu}}}{2.2^*_{\mu}} \int_{\Om}\int_{\Om} \frac{|u(x)|^{2^{*}_{\mu}}|u(y)|^{2^{*}_{\mu}}}{|x-y|^{\mu}}~dxdy\\& \ra -\infty \text{ as } t \ra \infty.
 	\end{align*}
 	Hence, we can choose $t_0>0$ such that $e:= t_0u$ such that (ii) follows. \QED
 \end{proof}
 \begin{Lemma}
 	Let $\la\in (0,\;aS\|f\|_{L^r}^{-1})$  and  $\{u_n\}$ be a $(PS)_c$ sequence for $J_\la$ with 
 	\[
 	c<  \frac{N-\mu+2}{2(2N-\mu)}(aS_{H,L})^{\frac{2N-\mu}{N-\mu+2}}. 
 	\]
 	Then $\{u_n\}$ has a convergent subsequence.
 \end{Lemma}
 \begin{proof}
 	Proof follows using the same assertions as in Proposition \ref{khprop2} up to  \eqref{kh22}. Since  by H\"older's inequality and Sobolev embedding, we have 
 	\begin{align}\label{kh24}
 	\int_{ \Om}f(x)|u|^2~dx\leq S^{-1}\|f\|_{L^r}\|u\|^2. 
 	\end{align}
 Taking account the fact that  $\la\in (0,\;aS\|f\|_{L^r}^{-1})$,  \eqref{kh24} and proceeding as in  proof of Proposition \ref{khprop2}, we get 
 	\begin{align*}
 	c=&  \lim_{n\ra \infty}J_\la(u_n)- \frac{1}{2\theta}\ld J_\la^{\prime}(u_n), u_n \rd \nonumber\\
 	& \geq  \bigg\{\frac{a(\theta-1)}{2\theta  }  w_{i_0}  + \left(\frac{1}{2}-\frac{1}{2\theta}\right)\left( a- \la S^{-1}\|f\|_{L^r} \right) \|u\|^2   + \left(\frac{1}{2\theta}-\frac{1}{2.2^*_{\mu}}\right) v_{i_0} \bigg\} \\
 	& \geq    \left(\frac{1}{2}-\frac{1}{2.2^*_{\mu}}\right) a w_{i_0}  \geq    \frac{N-\mu+2}{2(2N-\mu)}(aS_{H,L})^{\frac{2N-\mu}{N-\mu+2}} \nonumber,
 	\end{align*}
 	which is not possible. Therefore, compactness of the Palais-Smale sequence holds.\QED
 \end{proof}
 Let
  $c_\la := \ds \inf_{ u \in H_0^1(\Om)\setminus \{0\}} \sup_{r\geq 0}J_\la(ru)$ be the Mountain Pass level.
 \begin{Lemma}\label{khlem12}
 	Let $N\geq 4$ then there exists a $\tilde{\e}>0$ such that if $\e \in (0,\tilde{\e})$, we have
 	\begin{align*}
c_\la< \frac{N-\mu+2}{2(2N-\mu)}(aS_{H,L})^{\frac{2N-\mu}{N-\mu+2}}.
 	\end{align*} 
 \end{Lemma}
 \begin{proof}
 	Adopting the Same asymptotic analysis  as in Lemma \ref{khlem10} up to \eqref{kh14}, there exists a $r_0 \in (0,1)$ such that $
\ds 	\sup_{0\leq r\leq r_0}J_\la(r u_\e) < \frac{N-\mu+2}{2(2N-\mu)}(aS_{H,L})^{\frac{2N-\mu}{N-\mu+2}}$, for all $\la>0$ and 
 	\begin{align}\label{kh23}
 	\sup_{ r\geq r_0}	J_\la(r u_\e) \leq \frac{N-\mu+2}{2(2N-\mu)}\left(aS_{H,L}\right)^{\frac{2N-\mu}{N-\mu+2}}+C_1\e^{N-2} - \frac{m_f\la r_0^2}{2} \int_{B_{\de}(0)} |U_\e|^2~dx , 
 	\end{align}
 	Using \eqref{kh15}, we have
 	\begin{equation}
 	\begin{aligned}\label{kh26}
 	\int_{B_{\de}(0)} |U_\e|^2~dx & \geq  C_3 \left \{
 	\begin{array}{ll}
 	\e^{2} |\log\; \e|, & \text{ if }  N=4,\\
 	\e^{2}, &  \text{ if }  N>4 . \\
 	\end{array} 
 	\right.
 	\end{aligned}
 	\end{equation}
 	Using  \eqref{kh23} and \eqref{kh26}, it follows that 
 	\begin{equation*}
 	\begin{aligned}
 	\sup_{ r\geq r_0}	J_\la(r u_\e) &   \leq \frac{N-\mu+2}{2(2N-\mu)}\left(aS_{H,L}\right)^{\frac{2N-\mu}{N-\mu+2}}+C_1\e^{N-2}  - C_3 \la  \left\{
 	\begin{array}{ll}
 	\e^{2} |\log\; \e|, & \text{ if }  N=4,\\
 	\e^{2}, &  \text{ if }  N\geq 5 . \\
 	\end{array} 
 	\right.
 	\end{aligned}
 	\end{equation*}
 	In case of  $N= 4$  as  $\e\ra 0$ then $|\log \e| \ra \infty$, thus  we can choose $\e^*>0$ such that for every $0<\e<\e^*$ we have, $
 	C_1\e^{2}- C_3 \la \e^{2} |\log\; \e|< 0.$ In case of  $N>4$,   we can choose  $\e^{**}>0$ such that for every $0<\e<\e^{**}$ we have, $	C_1\e^{N-2}- C_3 \la \e^{2} < 0.$ Now define $\tilde{\e}= \min\{\e^*,\; \e^{**}\}$. Therefore, for all $\la>0$ and $\e \in (0, \; \tilde{\e})$ we have $
 	\ds \sup_{ r\geq 0}J_\la(r u_\e) < \frac{N-\mu+2}{2(2N-\mu)}(aS_{H,L})^{\frac{2N-\mu}{N-\mu+2}}$, then by definition of $c_\la$, we have $c_\la <  \frac{N-\mu+2}{2(2N-\mu)}(aS_{H,L})^{\frac{2N-\mu}{N-\mu+2}}$.\QED	 
 \end{proof}
\textbf{Proof of Theorem \ref{khthm3}:} 
 	From Lemmas \ref{khlem11}, \ref{khlem12} and \cite[Theorem 6.1]{struwe2}, we obtain the existence of a solution $u \in H_0^1(\Om) $ of $(P_\la)$. Using \cite[Lemma 4.4]{yangjmaa}, we have $u$ is a positive solution of $(P_\la)$. \QED

\end{document}